\documentclass[12pt,english]{amsart}
\usepackage[T1]{fontenc}
\usepackage[latin9]{inputenc}
\usepackage{geometry}
\usepackage{amsthm}
\usepackage{amsbsy}
\usepackage{amstext}
\usepackage{amssymb}

\makeatletter
\theoremstyle{plain}
\newtheorem{thm}{\protect\theoremname}
  \theoremstyle{definition}
  \newtheorem{defn}[thm]{\protect\definitionname}
  \theoremstyle{plain}
  \newtheorem{prop}[thm]{\protect\propositionname}
  \theoremstyle{plain}
  \newtheorem{cor}[thm]{\protect\corollaryname}

\newlength{\myArraycolsep}
\usepackage{tikz}
\usetikzlibrary{decorations.markings, arrows}

\usepackage{subfig}
\usepackage{chngcntr}
\counterwithout{figure}{section}

\newcommand*{\arrowovernarrow}{\mathrel{\mathop{\vcenter{\offinterlineskip\ialign{\hbox to\dimexpr4mm{##}\cr$\to$\cr\noalign{\kern-0.8mm}$\not\to$\cr}}}}}

\newcommand*{\arrowoverarrow}{\mathrel{\mathop{\vcenter{\offinterlineskip\ialign{\hbox to\dimexpr4mm{##}\cr$\to$\cr\noalign{\kern0.5mm}$\to$\cr}}}}}

\@ifundefined{showcaptionsetup}{}{%
 \PassOptionsToPackage{caption=false}{subfig}}
\usepackage{subfig}
\makeatother

\usepackage{babel}
  \providecommand{\corollaryname}{Corollary}
  \providecommand{\definitionname}{Definition}
  \providecommand{\propositionname}{Proposition}
\providecommand{\theoremname}{Theorem}

\begin{document}

\title{Zeta-equivalent digraphs: Simultaneous cospectrality}

\author{Peter Herbrich}
\begin{abstract}
We introduce a zeta function of digraphs that determines, and is determined~by,
the spectra of all linear combinations of the adjacency matrix, its
transpose, the out-degree matrix, and the in-degree matrix. In particular,
zeta-equivalence of graphs encompasses simultaneous cospectrality
with respect to the adjacency, the Laplacian, the signless Laplacian,
and the normalized Laplacian matrix. In addition, we express zeta-equivalence
in terms of Markov chains and in terms of invasions where each edge
is replaced by a fixed digraph. We finish with a method for constructing
zeta-equivalent digraphs.
\end{abstract}



\thanks{I am indebted to Peter Doyle for initiating this research and for
his indispensable contributions, and to Everett Sullivan for writing
parts of the code used to search for zeta-equivalent digraphs.}

\address{Dartmouth College, Hanover, New Hampshire, USA}

\email{peter.herbrich@dartmouth.edu}

\maketitle
\noindent 

\begin{minipage}[t]{1\columnwidth}%
\global\long\def\identityMatrix#1{I_{#1}}

\global\long\def\degreeSequence#1{\mathfrak{D}_{#1}}

\global\long\def\degreeMatrix#1{D_{#1}}

\global\long\def\adjacencyMatrix#1{A_{#1}}

\global\long\def\indegreeMatrix#1{D_{#1}^{\mathrm{in}}}

\global\long\def\outdegreeMatrix#1{D_{#1}^{\mathrm{out}}}

\global\long\def\variablesCBC{\boldsymbol{t}}

\global\long\def\variableCBC#1#2{t_{#1#2}}

\global\long\def\variablesL{\boldsymbol{u}}

\global\long\def\variableL#1{u_{#1}}

\global\long\def\backward{\downarrow}

\global\long\def\forward{\uparrow}

\global\long\def\rowVector#1{A_{#1}}

\global\long\def\adjointMatrix#1{\mathrm{adj}(#1)}

\global\long\def\nrEdges#1{m_{#1}}

\global\long\def\nrVertices#1{n_{#1}}

\global\long\def\graphVersion#1{\overset{\rightleftharpoons}{#1}}

\global\long\def\nrUndirectedEdges#1{\graphVersion m_{#1}}

\global\long\def\characteristicPolynomial#1{\chi_{#1}}

\global\long\def\characteristicPolynomialOfGraph#1{\graphVersion{\chi}_{#1}}

\global\long\def\generalizedCharacteristicPolynomial#1{\eta_{#1}}

\global\long\def\generalizedCharacteristicPolynomialOfGraph#1{\graphVersion{\eta}_{#1}}

\global\long\def\markovChainCharacteristicPolynomialOfGraphs#1{\graphVersion{\mu}_{#1}}

\global\long\def\graph{G}

\global\long\def\substituent{S}

\global\long\def\core{C}

\global\long\def\vertex{v}

\global\long\def\substituentTail{t}

\global\long\def\substituentHead{h}

\global\long\def\invadedDiraph#1#2{#1\succ#2}

\global\long\def\mate#1{#1'}

\global\long\def\weightMatrix#1{W_{#1}}

\global\long\def\weight#1{w(#1)}

\global\long\def\symPart#1{#1^{\mathrm{sym}}}

\global\long\def\edge{e}

\global\long\def\tail#1{t(#1)}

\global\long\def\head#1{h(#1)}

\global\long\def\length#1#2{|#1|_{#2}}

\global\long\def\path{\gamma}

\global\long\def\runsInto#1#2{#1\rightarrow#2}

\global\long\def\runsNotInto#1#2{#1\not\rightarrow#2}

\global\long\def\inverse#1#2{#1\rightleftarrows#2}

\global\long\def\cyclicBumpCount#1#2#3{b_{#1#2}(#3)}

\global\long\def\sign#1{\mathrm{dir}(#1)}

\global\long\def\bidirectional#1{#1_{\updownarrow}}

\global\long\def\tailIncidenceMatrix{T}

\global\long\def\headIncidenceMatrix{H}

\global\long\def\edgeNonBumpMatrix#1{C_{#1}}

\global\long\def\edgeBumpMatrix#1#2{B_{#1#2}}

\global\long\def\lyndonWord{\omega}

\global\long\def\lyndonWords{L}

\global\long\def\nrLoops#1{l_{#1}}

\global\long\def\nrPairsOfReverseEdges#1{r_{#1}}

\global\long\def\variableProduct#1{s_{#1}}

\global\long\def\invadedGraph#1#2{#1\succeq#2}

\global\long\def\allOnesMatrix#1{J_{#1}}

\global\long\def\complementDiGraph#1{#1^{\mathrm{c}}}

\global\long\def\subgraphIsomorphism{\Phi}

\global\long\def\vertices#1{V_{#1}}

\global\long\def\vParts#1{V_{#1}}

\global\long\def\wParts#1{W_{#1}}

\global\long\def\xPart{X}

\global\long\def\wertex{w}

\global\long\def\xertex{x}

\global\long\def\nrVParts{p}

\global\long\def\nrWParts{q}

\global\long\def\nrEdgeFromXtoY#1#2{m_{#1\to#2}}

\global\long\def\block#1#2{B_{#1#2}}

\global\long\def\unlinked#1#2{#1\not\to#2}

\global\long\def\halflinked#1#2{#1\arrowovernarrow#2}
\global\long\def\fullylinked#1#2{#1\arrowoverarrow#2}

\global\long\def\difference#1{\delta_{#1}}

\global\long\def\outDegree#1{d_{#1}^{\mathrm{out}}}

\global\long\def\vvPrimeParts#1{\vParts{#1}^{*}}

\global\long\def\diagonalMatrix#1{\mathrm{diag}\left(#1\right)}
\end{minipage}

\section{Introduction\label{sec:Introduction}}

\noindent We study digraphs which either have parallel edges or weighted
edges. If $\graph$ is a digraph of one of these two types, we denote
its number of vertices by~$\nrVertices{\graph}$, its number of edges
by~$\nrEdges{\graph}$, its adjacency matrix by~$\adjacencyMatrix{\graph}\in\mathbb{Z}_{\geq0}^{\nrVertices{\graph}\times\nrVertices{\graph}}$,
its out-degree matrix by $\outdegreeMatrix{\graph}$, and its in-degree
matrix by $\indegreeMatrix{\graph}$, the latter two of which have
the row sums of $\adjacencyMatrix{\graph}$ and $\adjacencyMatrix{\graph}^{T}$
on their diagonals. We call $\graph$ an undirected digraph, or simply
a graph, if $\adjacencyMatrix{\graph}^{T}=\adjacencyMatrix{\graph}$
so that $\degreeMatrix{\graph}=\outdegreeMatrix{\graph}=\indegreeMatrix{\graph}$.
\begin{defn}
The generalized characteristic polynomial of a digraph $\graph$ reads
\[
\generalizedCharacteristicPolynomial{\graph}(x,\variableCBC{\forward}{},\variableCBC{\backward}{},\variableL{\forward},\variableL{\backward})=\det(x\identityMatrix{\nrVertices{\graph}}+\variableCBC{\forward}{}\outdegreeMatrix{\graph}+\variableCBC{\backward}{}\indegreeMatrix{\graph}+\variableL{\forward}\adjacencyMatrix{\graph}+\variableL{\backward}\adjacencyMatrix{\graph}^{T}).
\]
If $\graph$ is a graph, then
\[
\generalizedCharacteristicPolynomialOfGraph{\graph}(x,\variableCBC{}{},\variableL{})=\generalizedCharacteristicPolynomial{\graph}(x,\variableCBC{}{},0,\variableL{},0)=\det(x\identityMatrix{\nrVertices{\graph}}+\variableCBC{}{}\degreeMatrix{\graph}+\variableL{}\adjacencyMatrix{\graph}).
\]

\end{defn}
The definition of $\generalizedCharacteristicPolynomialOfGraph{\graph}$
goes back to~\cite{CvetkovicDoobSachs1980}. Note that $\generalizedCharacteristicPolynomial{\graph}$
and $\generalizedCharacteristicPolynomialOfGraph{\graph}$ are homogeneous
and generalize the characteristic polynomial $\characteristicPolynomial{\graph}(x)=\characteristicPolynomialOfGraph{\graph}(x)=\generalizedCharacteristicPolynomialOfGraph{\graph}(x,0,-1)$.
If $\graph$ is a graph, then the single-variable polynomials $\generalizedCharacteristicPolynomialOfGraph{\graph}(x,-1,1)$,
$\generalizedCharacteristicPolynomialOfGraph{\graph}(x,-1,-1)$, and
$\generalizedCharacteristicPolynomialOfGraph{\graph}(0,t,-1)$ determine
the spectra of the Laplacian matrix $\degreeMatrix{\graph}-\adjacencyMatrix{\graph}$,
the signless Laplacian matrix $\degreeMatrix{\graph}+\adjacencyMatrix{\graph}$,
and the normalized Laplacian matrix $\degreeMatrix{\graph}^{-1/2}(\degreeMatrix{\graph}-\adjacencyMatrix{\graph})\degreeMatrix{\graph}^{-1/2}$,
respectively, where $\degreeMatrix{\graph}$ is assumed to be invertible
in the last case.

First, let $\graph$ be a digraph with weighted edges, none of which
are parallel to each other. In Section~\ref{sec:Zeta_functions},
we introduce an auxiliary digraph $\bidirectional{\graph}$ that allows
to study closed walks on the vertices of $\graph$, where at each
step either an outgoing or an incoming edge is traversed. Every such
walk $\path$ inherits a weight $\weight{\path}\in\mathbb{C}\backslash\{0\}$
from the edges it involves. We let $\length{\path}{\forward}$ and
$\length{\path}{\backward}$ be its directional lengths which are
the numbers of outgoing and incoming edges traversed, respectively.
Moreover, we let $\cyclicBumpCount{\forward}{\forward}{\path}$ be
one of its cyclic bump counts given as the number of times two outgoing
edges with opposite directions are used consecutively, where the last
and first edge are considered consecutive. The cyclic bump counts
$\cyclicBumpCount{\backward}{\backward}{\path}$, $\cyclicBumpCount{\forward}{\backward}{\path}$,
and $\cyclicBumpCount{\backward}{\forward}{\path}$ are defined similarly.
The directional lengths and cyclic bump counts of a closed walk $\path$
descend to its equivalence class $[\path]$ under cyclic permutation
of edges, called its cycle. Lastly, a cycle is called primitive if
none of its representatives is a multiple of a strictly shorter closed
walk.
\begin{defn}
\label{def:Zeta_function}Let $\mathcal{P}_{\graph}$ and $\mathcal{P}_{\bidirectional{\graph}}$
denote the set of primitive cycles in $\graph$ and $\bidirectional{\graph}$,
respectively. The zeta function of $\graph$ with arguments $\variablesCBC=(\variableCBC{\forward}{\forward},\variableCBC{\backward}{\backward},\variableCBC{\forward}{\backward},\variableCBC{\backward}{\forward})$
and $\variablesL=(\variableL{\forward},\variableL{\backward})$ is
given for sufficiently small $|\variablesCBC|$ and $|\variablesL|$
by 
\[
\zeta_{\graph}(\variablesCBC,\variablesL)=\prod_{[\path]\in\mathcal{P}_{\bidirectional{\graph}}}\frac{1}{1-\weight{\path}\variableCBC{\forward}{\forward}^{\cyclicBumpCount{\forward}{\forward}{\path}}\variableCBC{\backward}{\backward}^{\cyclicBumpCount{\backward}{\backward}{\path}}\variableCBC{\forward}{\backward}^{\cyclicBumpCount{\forward}{\backward}{\path}}\variableCBC{\backward}{\forward}^{\cyclicBumpCount{\backward}{\forward}{\path}}\variableL{\forward}^{\length{\path}{\forward}}\variableL{\backward}^{\length{\path}{\backward}}}.
\]
The reversing zeta function is obtained by setting $\variableCBC{\forward}{\forward}=\variableCBC{\downarrow}{\downarrow}=1$,
namely, 
\[
\zeta_{\graph}^{\forward\backward}(\variableCBC{\forward}{},\variableCBC{\backward}{},\variableL{\forward},\variableL{\backward})=\zeta_{\graph}(1,1,\variableCBC{\forward}{},\variableCBC{\backward}{},\variableL{\forward},\variableL{\backward}).
\]
The outgoing zeta function is obtained by setting $\variableL{\backward}=0$,
namely,
\[
\zeta_{\graph}^{\forward\forward}(\variableCBC{}{},\variableL{})=\zeta_{\graph}(\variableCBC{}{},0,0,0,\variableL{},0)=\zeta_{\graph}(\variableCBC{}{},1,1,1,\variableL{},0)=\prod_{[\path]\in\mathcal{P}_{\graph}}\frac{1}{1-\weight{\path}\variableCBC{}{}^{\cyclicBumpCount{\forward}{\forward}{\path}}\variableL{}^{\length{\path}{\forward}}}.
\]

\end{defn}
Among other things, we show that $\zeta_{\graph}^{-1}$ is a polynomial.
Choe~et~al.~\cite{ChoeKwakParkSato2007} considered $\zeta_{\graph}^{\forward\forward}$
for weighted digraphs without loops. As a by-product, we provide an
alternative derivation of the corresponding polynomial expression
of $(\zeta_{\graph}^{\forward\forward})^{-1}$. If $\graph$ is an
unweighted graph, then $\zeta_{\graph}^{\forward\forward}$ reduces
to the zeta function introduced by Bartholdi~\cite{Bartholdi1999},
which in turn generalizes the mother of all graph zeta functions,
the Ihara-Selberg zeta function~\cite{Ihara1966} 
\[
\zeta_{\mathrm{IS}}(\variableL{})=\zeta_{\graph}^{\forward\forward}(0,\variableL{})=\zeta_{\graph}(0,0,0,0,\variableL{},0).
\]
Since cycles with bumps contribute a negligible factor of $1$ to
the product expression of $\zeta_{\mathrm{IS}}$, the original definition
of $\zeta_{\mathrm{IS}}$ involved reduced cycles, none of whose representatives
has a bump. We point out the pioneering works of Ihara~\cite{Ihara1966}
and Bass~\cite{Bass1992}, who found the polynomial expression of
$\zeta_{\mathrm{IS}}^{-1}$ for regular and non-regular graphs, respectively.

Now, let $\graph$ be an unweighted digraph, possibly with parallel
edges. In Section~\ref{sec:Invasions}, we introduce invaded digraphs
$\invadedDiraph{\substituent}{\graph}$ that arise from $\graph$
by replacing each of its edges by an invader $\substituent$ which
is a digraph with two distinguished vertices $\substituentTail$ and
$\substituentHead$. If $\graph$ is a graph and $\substituent$ is
symmetric meaning it has an automorphism that interchanges $\substituentTail$
and $\substituentHead$, then the symmetrically invaded digraph $\invadedGraph{\substituent}{\graph}$
arises by replacing each undirected edge of $\graph$ by $\substituent$.

Finally, let $\graph$ be an unweighted graph, possibly with parallel
edges. For every pair of positive integers $a,b\in\mathbb{Z}_{>0}$,
let $\graph_{a,b}$ be the digraph obtained by attaching $a$ parallel
loops to every vertex of $\graph$, adding a vertex $\vertex$ with
a single loop to $\graph$, and adding $b$ parallel directed edges
from every vertex of $\graph$ to $\vertex$. We consider the Markov
process on the vertices of $\graph_{a,b}$, where at each step one
of the outgoing edges is chosen with uniform probability. The spectrum
of the corresponding transition matrix is determined by the $a$-lazy
$b$-deadly Markov chain function defined as follows.
\begin{defn}
The $a$-lazy $b$-deadly Markov chain function is given for $a,b>0$
by
\[
\markovChainCharacteristicPolynomialOfGraphs{\graph}(x,a,b)=\det(x\identityMatrix{\nrVertices{\graph}}+((a+b)\identityMatrix{\nrVertices{\graph}}+\degreeMatrix{\graph})^{-1}(\adjacencyMatrix{\graph}+a\identityMatrix{\nrVertices{\graph}})).
\]

\end{defn}
The aforementioned definitions apply to simple unweighted digraphs
and graphs, for which they lead to the same equivalence relation as
stated by the following main theorems. Their proofs are given in Sections~\ref{sec:Zeta_functions}
to~\ref{sec:Markov_chains}, each of which contains results of independent
interest.
\begin{thm}
\label{thm:Zeta_equivalent_digraphs}If $\graph$ and $\mate{\graph}$
are simple digraphs with $\nrVertices{\graph}=\nrVertices{\mate{\graph}}$,
then the following are equivalent:\negthinspace{}\negthinspace{}
\begin{enumerate}
\item $\zeta_{\graph}^{\forward\backward}=\zeta_{\mate{\graph}}^{\forward\backward}$,
\item $\generalizedCharacteristicPolynomial{\graph}=\generalizedCharacteristicPolynomial{\mate{\graph}}$,
\item $\generalizedCharacteristicPolynomial{\invadedDiraph{\substituent}{\graph}}=\generalizedCharacteristicPolynomial{\invadedDiraph{\substituent}{\mate{\graph}}}$
for every invader $\substituent$,
\item $\characteristicPolynomial{\invadedDiraph{\substituent}{\graph}}=\characteristicPolynomial{\invadedDiraph{\substituent}{\mate{\graph}}}$
for every invader $\substituent$,
\item $\characteristicPolynomial{\invadedDiraph{\substituent}{\graph}}=\characteristicPolynomial{\invadedDiraph{\substituent}{\mate{\graph}}}$
for every invader $\substituent$ with $\adjacencyMatrix{\substituent}\in\mathbb{Z}_{>0}^{2\times2}$.
\end{enumerate}
\end{thm}

\begin{thm}
\label{thm:Zeta_equivalent_graphs}If $\graph$ and $\mate{\graph}$
are simple graphs with $\nrVertices{\graph}=\nrVertices{\mate{\graph}}$,
then the following are equivalent:
\begin{enumerate}
\item $\zeta_{\graph}^{\forward\forward}=\zeta_{\mate{\graph}}^{\forward\forward}$,
\item $\markovChainCharacteristicPolynomialOfGraphs{\graph}=\markovChainCharacteristicPolynomialOfGraphs{\mate{\graph}}$,
\item $\markovChainCharacteristicPolynomialOfGraphs{\graph|b=c}=\markovChainCharacteristicPolynomialOfGraphs{\mate{\graph}|b=c}$
for some $c\in\mathbb{C}$.
\item $\generalizedCharacteristicPolynomialOfGraph{\graph}=\generalizedCharacteristicPolynomialOfGraph{\mate{\graph}}$,\label{enu:generalized_characteristic_polynomials_equal}
\item $\generalizedCharacteristicPolynomialOfGraph{\invadedGraph{\substituent}{\graph}}=\generalizedCharacteristicPolynomialOfGraph{\invadedGraph{\substituent}{\mate{\graph}}}$
for every symmetric invader $\substituent$,
\item $\characteristicPolynomialOfGraph{\invadedGraph{\substituent}{\graph}}=\characteristicPolynomialOfGraph{\invadedGraph{\substituent}{\mate{\graph}}}$
for every symmetric invader $\substituent$,
\item $\characteristicPolynomialOfGraph{\invadedGraph{\substituent}{\graph}}=\characteristicPolynomialOfGraph{\invadedGraph{\substituent}{\mate{\graph}}}$
for every undirected path $\substituent$ with endpoints $\substituentTail$
and $\substituentHead$,
\item $\characteristicPolynomialOfGraph{\invadedGraph{\substituent}{\graph}}=\characteristicPolynomialOfGraph{\invadedGraph{\substituent}{\mate{\graph}}}$
for every $\substituent$ with $\adjacencyMatrix{\substituent}=\left(\begin{array}{cc}
k & 1\\
1 & k
\end{array}\right)\in\mathbb{Z}_{>0}^{2\times2}$.
\item $\characteristicPolynomialOfGraph{\invadedGraph{\substituent}{\graph}}=\characteristicPolynomialOfGraph{\invadedGraph{\substituent}{\mate{\graph}}}$
for every $\substituent$ with $\adjacencyMatrix{\substituent}=\left(\begin{array}{cc}
1 & k\\
k & 1
\end{array}\right)\in\mathbb{Z}_{>0}^{2\times2}$.\label{enu:characteristic_polynomials_of_2_by_2_invasions_equal}
\end{enumerate}
\end{thm}
\begin{defn}
Two simple digraphs (or graphs) $\graph$ and $\mate{\graph}$ are
called zeta-equivalent if they satisfy any of the equivalent conditions
in Theorem~\ref{thm:Zeta_equivalent_digraphs} (or Theorem~\ref{thm:Zeta_equivalent_graphs}).
\end{defn}
Durfee and Martin~\cite{DurfeeMartin2014} enumerated all zeta-equivalent
graphs with up to 11 vertices. For regular graphs, zeta-equivalence
coincides with cospectrality with respect to the adjacency matrix,
Laplacian matrix, or signless Laplacian matrix, respectively. In contrast,
a non-regular graph $\graph$ may be uniquely determined by $\generalizedCharacteristicPolynomialOfGraph{\graph}$
as studied in~\cite{WangLiLuXu2011,DurfeeMartin2014}, but not by
some of the single-variable polynomials $\generalizedCharacteristicPolynomialOfGraph{\graph}(x,0,-1)$,
$\generalizedCharacteristicPolynomialOfGraph{\graph}(x,-1,1)$, and
$\generalizedCharacteristicPolynomialOfGraph{\graph}(x,-1,-1)$. It
is known that almost all trees are zeta-equivalent to some other tree~\cite{Osborne2013},
yet if $\graph$ and $\mate{\graph}$ are non-isomorphic trees, then
there exists a two-variable real polynomial $p$ such that $\det(x\identityMatrix{\nrVertices{\graph}}+p(\adjacencyMatrix{\graph},\degreeMatrix{\graph}))\neq\det(x\identityMatrix{\nrVertices{\mate{\graph}}}+p(\adjacencyMatrix{\mate{\graph}},\degreeMatrix{\mate{\graph}}))$~\cite{McKay1977}.
The proof of the former statement is largely the same as in~\cite{Schwenk1973},
which settles the case of adjacency cospectral trees.

Figure~\ref{fig:Smallest_zeta_equivalent_weekly_connected_digraphs}
shows the smallest zeta-equivalent (weakly) connected simple (di)graphs.
The graphs in Figure~\ref{fig:Zeta_equivalent_complementary_graphs}
arise from a method which we introduce in Section~\ref{sec:Digraph-switchings}.
The digraphs in Figure~\ref{fig:Zeta_equivalent_digraphs} are zeta-equivalent
since their generalized Laplacian matrices $\variableCBC{\forward}{}\outdegreeMatrix{\graph}+\variableCBC{\backward}{}\indegreeMatrix{\graph}+\variableL{\forward}\adjacencyMatrix{\graph}+\variableL{\backward}\adjacencyMatrix{\graph}^{T}$
are conjugated by
\[
\left(\begin{array}{ccccc}
\variableL{\forward}^{2}-\variableCBC{\backward}{}\variableL{\backward} & \variableL{\forward}\variableL{\backward} & 0 & 0 & 0\\
\variableL{\backward}^{2} & \variableL{\forward}^{2}+\variableCBC{\backward}{}\variableL{\backward} & \variableL{\backward}^{2} & \variableL{\forward}\variableL{\backward} & 0\\
0 & \variableL{\backward}^{2} & \variableL{\forward}\variableL{\backward} & \variableL{\forward}^{2} & 0\\
0 & \variableL{\forward}\variableL{\backward} & \variableL{\forward}^{2} & \variableL{\backward}^{2}+\variableCBC{\forward}{}\variableL{\forward} & \variableL{\forward}^{2}\\
0 & 0 & 0 & \variableL{\forward}\variableL{\backward} & \variableL{\backward}^{2}-\variableL{\forward}\variableCBC{\forward}{}
\end{array}\right),
\]
which has determinant $(\variableL{\forward}^{5}-2\variableL{\forward}^{2}\variableL{\backward}^{3}-\variableCBC{\backward}{}^{2}\variableL{\forward}\variableL{\backward}^{2}+\variableCBC{\backward}{}\variableL{\backward}^{4})(\variableL{\backward}^{5}-2\variableL{\backward}^{2}\variableL{\forward}^{3}-\variableCBC{\forward}{}^{2}\variableL{\backward}\variableL{\forward}^{2}+\variableCBC{\forward}{}\variableL{\forward}^{4})$.
It is worth mentioning that the adjacency matrices of these digraphs
have different Jordan normal forms, which is why the conjugating matrix
has to become singular at $\variableCBC{\forward}{}=\variableCBC{\backward}{}=\variableL{\backward}=0$.

\noindent \newcommand{\smallestDigraphs}[1]{
 \def\flip{#1}
 \tikzstyle{every node}=
  [circle, draw, fill=black!50, inner sep=1pt, minimum width=3pt]

\begin{tikzpicture}[scale=0.35,>=triangle 45,thick,decoration={
   markings,mark=at position 0.6 with {\arrow{>}}}] 
  \draw[postaction={decorate}] ( 0, 0) -- (\flip * -2,-2);
  \draw[postaction={decorate}] (\flip * -2,-2) -- (\flip * 2,-2);
  \draw[postaction={decorate}] (\flip * 2,-2) -- ( 0, 0);
  \draw[postaction={decorate}] (-2, 2) -- (-2,-2);
  \draw[postaction={decorate}] ( 2,-2) -- ( 2, 2);

  \draw (-2, 2) node {};
  \draw (-2, 2) node[left=2pt, draw=none, fill=white] {$1$};
  \draw (-2,-2) node {};
  \draw (-2,-2) node[left=2pt, draw=none, fill=white] {$2$};
  \draw ( 0, 0) node {};
  \draw ( 0, 0) node[above=2pt, draw=none, fill=white] {$3$};
  \draw ( 2,-2) node {};
  \draw ( 2,-2) node[right=2pt, draw=none, fill=white] {$4$};
  \draw ( 2, 2) node {};
  \draw ( 2, 2) node[right=2pt, draw=none, fill=white] {$5$};
 \end{tikzpicture}
}

\noindent \newcommand{\smallestGraphs}[1]{
 \def\flip{#1}
 \tikzstyle{every node}=
  [circle, draw, fill=black!50, inner sep=1pt, minimum width=3pt]

 \begin{tikzpicture}[thick,scale=0.4]
   \draw (-3,-1) node[below left=2pt, draw=none, fill=white] {$5$};
   \draw (-1,-1) node[below left=2pt, draw=none, fill=white] {$1$};
   \draw ( 1,-1) node[below right=2pt, draw=none, fill=white] {$2$};
   \draw ( 3,-1) node[below right=2pt, draw=none, fill=white] {$6$};
   \draw (-3, 1) node[above left=2pt, draw=none, fill=white] {$7$};
   \draw (-1, 1) node[above left=2pt, draw=none, fill=white] {$3$};
   \draw ( 1, 1) node[above right=2pt, draw=none, fill=white] {$4$};
   \draw ( 3, 1) node[above right=2pt, draw=none, fill=white] {$8$};
   \draw ( 0,\flip * 2) node[above, draw=none, fill=white] {$9$};

  \draw ( 0, \flip * 2) -- (-1, \flip * 1);
  \draw ( 0, \flip * 2) -- ( 1, \flip * 1);
  \draw ( 0, \flip * 2) .. controls ( 2.5, \flip * 2) .. ( 3, \flip * 1);
  \draw ( 0, \flip * 2) .. controls (-2.5, \flip * 2) .. (-3, \flip * 1);
  \draw (-3,-1) -- (-1,1);
  \draw (-3, 1) -- (1,-1);
  \draw (-1,-1) -- (3, 1);
  \draw ( 1, 1) -- (3,-1);

  \draw \foreach \x in {-3,-1,...,1}
  { (\x,-1) -- (\x + 2, -1)
    (\x, 1) -- (\x + 2,  1)
    (\x,-1) node {} -- (\x,  1) node {}
  };
  \draw (3,-1) node {} -- (3,1) node {};
  \draw (0, \flip * 2) node {};
  \draw (0, \flip * -2) node[draw=none, fill=white] {};
 \end{tikzpicture}
}

\begin{figure}
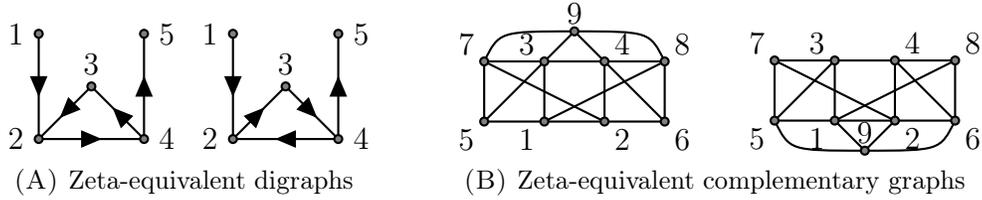

\noindent \begin{centering}
\subfloat[Zeta-equivalent digraphs\label{fig:Zeta_equivalent_digraphs}]{\noindent \begin{centering}
\smallestDigraphs{1}\hspace{-1em}\smallestDigraphs{-1}
\par\end{centering}

}\subfloat[Zeta-equivalent complementary graphs\label{fig:Zeta_equivalent_complementary_graphs}]{\noindent \begin{centering}
\smallestGraphs{1}\smallestGraphs{-1}
\par\end{centering}

}
\par\end{centering}

\caption{Smallest zeta-equivalent (weakly) connected simple (di)graphs. The
right-hand pair was found by Durfee and Martin~\cite{DurfeeMartin2014}.
It is also the smallest pair of cospectral graphs with cospectral
line graphs, cospectral complements, and complements that have cospectral
line graphs~\cite{GodsilMcKay1976}.\label{fig:Smallest_zeta_equivalent_weekly_connected_digraphs}}
\end{figure}

\section{Zeta functions\label{sec:Zeta_functions}}

\noindent Let $\graph$ be a digraph with weighted non-parallel edges.
We let $\weightMatrix{\graph}\in\mathbb{C}^{\nrVertices{\graph}\times\nrVertices{\graph}}$
denote the weighted adjacency matrix of $\graph$ which arises form
$\adjacencyMatrix{\graph}\in\{0,1\}^{\nrVertices{\graph}\times\nrVertices{\graph}}$
by replacing each $1$ by the weight $\weight{\edge}\in\mathbb{C}\backslash\{0\}$
of the corresponding directed edge $\edge$ of $\graph$. Moreover,
we let~$\symPart{\adjacencyMatrix{\graph}}\in\{0,1\}^{\nrVertices{\graph}\times\nrVertices{\graph}}$
be the symmetric part of $\adjacencyMatrix{\graph}$ with entries
$[\symPart{\adjacencyMatrix{\graph}}]_{ij}=[\adjacencyMatrix{\graph}]_{ij}[\adjacencyMatrix{\graph}]_{ji}$,
$\symPart{\weightMatrix{\graph}}\in\mathbb{C}^{\nrVertices{\graph}\times\nrVertices{\graph}}$
be the corresponding part of $\weightMatrix{\graph}$ with entries
$[\symPart{\weightMatrix{\graph}}]_{ij}=[\symPart{\adjacencyMatrix{\graph}}]_{ij}[\weightMatrix{\graph}]_{ij}$,
and $\symPart{\degreeMatrix{\graph}}$ be the diagonal matrix that
has the row sums of $\symPart{\adjacencyMatrix{\graph}}$ on its diagonal.

If $\edge=(\vertex,\vertex')$ is an edge of $\graph$, we call $\tail{\edge}=\vertex$
and $\head{\edge}=\vertex'$ the tail and head of $\edge$, respectively.
We write $\runsInto{\edge}{\edge'}$ whenever $\head{\edge}=\tail{\edge'}$.
A path $\path$ in $\graph$ of length $\length{\path}{}\in\mathbb{Z}_{>0}$
is a  sequence of edges $\path=(\edge_{1},\edge_{2},\ldots,\edge_{\length{\path}{}})$
with $\runsInto{\edge_{i}}{\edge_{i+1}}$ for all $i\in\{1,2,\ldots,\length{\path}{}-1\}$.
We call $\weight{\path}=\weight{\edge_{1}}\weight{\edge_{2}}\cdots\weight{\edge_{\length{\path}{}}}$
the weight of $\path$, and let $\tail{\path}=\tail{\edge_{1}}$ as
well as $\head{\path}=\head{\edge_{\length{\path}{}}}$. A circuit
is a path $\path$ with $\head{\path}=\tail{\path}$, and the equivalence
class under cyclic permutation of its edges is called a cycle, denoted
by $[\path]$. If $\path$ is a circuit and $k\in\mathbb{Z}_{>0}$,
we let $\path^{k}$ denote the $k$-fold concatenation with itself,
in particular, $\length{\path^{k}}{}=k\length{\path}{}$ and $\weight{\path^{k}}=(\weight{\path})^{k}$.
A cycle is called primitive if none of its representatives is of the
form $\path^{k}$ with $k\geq2$.

If $\edge$ is an edge, then $\edge'=(\head{\edge},\tail{\edge})$
is said to be reverse to $\edge$, which is indicated by $\inverse{\edge'}{\edge}$.
The graph $\graph$ is said to have reciprocal weights if $\weight{\edge'}=\weight{\edge}^{-1}$
whenever $\inverse{\edge'}{\edge}$. In this case, $\weight{\edge}=\pm1$
if $\edge$ is a loop. Let $\bidirectional{\graph}$ denote the digraph
obtained by adding a reverse edge of reciprocal weight to each edge
of $\graph$. Note that $\bidirectional{\graph}$ has pairs of parallel
edges if $\graph$ has loops or pairs of mutually reverse edges. We
define the direction of edges $\edge$ of $\bidirectional{\graph}$
as 
\[
\sign{\edge}=\begin{cases}
\forward & \text{if }\edge\text{ belongs to }\graph,\\
\backward & \text{if }\edge\text{ was added.}
\end{cases}
\]
To each circuit $\path=(\edge_{1},\edge_{2},\ldots,\edge_{\length{\path}{}})$
in $\bidirectional{\graph}$, we associate $4$ non-negative integers,
called its cyclic bump counts. For $d,d'\in\{\forward,\backward\}$,
we define
\begin{eqnarray*}
\cyclicBumpCount d{d'}{\path} & = & |\{i=1,\ldots,\length{\path}{}\mid\sign{\edge_{i}}=d,\sign{\edge_{i+1}}=d',\inverse{\edge_{i}}{\edge_{i+1}}\}|,
\end{eqnarray*}
where $\edge_{\length{\path}{}+1}=\edge_{1}$. Moreover, we define
directional lengths as
\[
\length{\path}d=|\{i=1,\ldots,\length{\path}{}\mid\sign{\edge_{i}}=d\}|,\text{ so that }\length{\path}{}=\length{\path}{\forward}+\length{\path}{\backward}.
\]
As noted in Section~\ref{sec:Introduction}, weights, cyclic bump
counts, and directional lengths are well-defined for cycles. We encode
these quantities in terms of matrices associated with the line graph
of~$\bidirectional{\graph}$. For $d,d'\in\{\forward,\backward\}$,
let $\edgeBumpMatrix d{d'}\in\mathbb{\mathbb{C}}^{2\nrEdges{\graph}\times2\nrEdges{\graph}}$
be the weighted bump matrix given by 
\[
[\edgeBumpMatrix d{d'}]_{\edge\edge'}=\begin{cases}
\weight{\edge} & \text{if }\sign{\edge}=d,\sign{\edge'}=d',\text{ and }\inverse{\edge}{\edge'},\\
0 & \text{otherwise,}
\end{cases}
\]
and let $\edgeNonBumpMatrix d,M(\variablesCBC,\variablesL)\in\mathbb{\mathbb{C}}^{2\nrEdges{\graph}\times2\nrEdges{\graph}}$
be the weighted adjacency matrices given by
\[
[\edgeNonBumpMatrix d]_{\edge\edge'}=\begin{cases}
\weight{\edge} & \text{if }\sign{\edge}=d\text{ and }\runsInto{\edge}{\edge'},\\
0 & \text{otherwise},
\end{cases}
\]
and
\[
M(\variablesCBC,\variablesL)=\variableL{\forward}(\edgeNonBumpMatrix{\forward}+(\variableCBC{\forward}{\forward}-1)\edgeBumpMatrix{\forward}{\forward}+(\variableCBC{\forward}{\backward}-1)\edgeBumpMatrix{\forward}{\backward})+\variableL{\backward}(\edgeNonBumpMatrix{\backward}+(\variableCBC{\backward}{\backward}-1)\edgeBumpMatrix{\backward}{\backward}+(\variableCBC{\backward}{\forward}-1)\edgeBumpMatrix{\backward}{\forward}),
\]
where we reused the abbreviations $\variablesCBC=(\variableCBC{\forward}{\forward},\variableCBC{\backward}{\backward},\variableCBC{\forward}{\backward},\variableCBC{\backward}{\forward})$
and $\variablesL=(\variableL{\forward},\variableL{\backward})$ from
Definition~\ref{def:Zeta_function}. In order to translate cycles
in $\bidirectional{\graph}$ to closed walks on the vertices of $\graph$,
we let $\tailIncidenceMatrix\in\mathbb{\mathbb{C}}^{\nrVertices{\graph}\times2\nrEdges{\graph}}$
be the tail-incidence matrix of $\bidirectional{\graph}$ given by
\[
[\tailIncidenceMatrix]_{\vertex\edge}=\begin{cases}
1 & \tail e=\vertex,\\
0 & \text{otherwise.}
\end{cases}
\]

\begin{thm}
The function $\zeta_{\graph}^{-1}$ is a polynomial in the components
of $\variablesCBC$ and $\variablesL$, namely,
\[
\zeta_{\graph}^{-1}(\variablesCBC,\variablesL)=\det(\identityMatrix{2\nrEdges{\graph}}-M(\variablesCBC,\variablesL)).
\]
If $\graph$ has reciprocal weights, then 
\begin{eqnarray}
\zeta_{\graph}^{-1}(1,1,\variableCBC{\forward}{\backward},\variableCBC{\backward}{\forward},\variableL{\forward},\variableL{\backward}) & = & (1-\variableProduct{\forward}\variableProduct{\backward})^{\nrEdges{\graph}-\nrVertices{\graph}}\det((1-\variableProduct{\forward}\variableProduct{\backward})\identityMatrix{\nrVertices{\graph}}-\variableL{\forward}\weightMatrix{\graph}-\variableL{\backward}\weightMatrix{\graph}^{*}\nonumber \\
 &  & \hphantom{(1-\variableProduct{\forward}\variableProduct{\backward})^{\nrEdges{\graph}-\nrVertices{\graph}}\det((1-\variableProduct{\forward}}-\variableProduct{\forward}\variableL{\backward}\outdegreeMatrix{\graph}-\variableProduct{\backward}\variableL{\forward}\indegreeMatrix{\graph}),\label{eq:Zeta_with_fb_or_bf_bump_count}
\end{eqnarray}
where $\variableProduct{\forward}=\variableL{\forward}(\variableCBC{\forward}{\backward}-1)$,
$\variableProduct{\backward}=\variableL{\backward}(\variableCBC{\backward}{\forward}-1)$,
and $\weightMatrix{\graph}^{*}$ denotes the weighted adjacency matrix
of the added edges in $\bidirectional{\graph}$, i.e., $\weightMatrix{\graph}^{*}$
equals $\weightMatrix{\graph}^{T}$ up to inversion of its nonzero
entries, and
\begin{eqnarray}
\zeta_{\graph}^{-1}(\variableCBC{\forward}{\forward},1,1,1,\variableL{\forward},\variableL{\backward}) & = & (1-\variableProduct{\forward})^{\nrLoops{\graph}^{+}}(1+\variableProduct{\forward})^{\nrLoops{\graph}^{-}}(1-\variableProduct{\forward}^{2})^{\nrPairsOfReverseEdges{\graph}-\nrVertices{\graph}}\det((1-\variableProduct{\forward}^{2})\identityMatrix{\nrVertices{\graph}}\nonumber \\
 &  & -(1-\variableProduct{\forward}^{2})(\variableL{\forward}\weightMatrix{\graph}+\variableL{\backward}\weightMatrix{\graph}^{*})-\variableProduct{\forward}\variableL{\forward}\symPart{\degreeMatrix{\graph}}-\variableProduct{\forward}^{2}\variableL{\forward}\symPart{\weightMatrix{\graph}}),\label{eq:Zeta_with_ff_or_bb_bump_count}
\end{eqnarray}
where $\variableProduct{\forward}=\variableL{\forward}(\variableCBC{\forward}{\forward}-1)$,
$\nrLoops{\graph}^{\pm}=\mathrm{tr}(\adjacencyMatrix{\graph}\pm\weightMatrix{\graph})/2$,
and $\nrPairsOfReverseEdges{\graph}=\mathrm{tr}(\adjacencyMatrix{\graph}^{2})/2-\mathrm{tr}(\adjacencyMatrix{\graph})$,
the latter of which are the number of loops with weight $\pm1$ and
pairs of mutually reverse edges of $\graph$, respectively.\end{thm}
\begin{proof}
Using ideas from~\cite{Hashimoto1989,Hashimoto1992,Bass1992,FoataZeilberger1999,ChoeKwakParkSato2007,Bartholdi2008},
we let $\lyndonWords$ be the set of Lyndon words of the ordered set
$X=\{1,2,\ldots,2\nrEdges{\graph}\}$, i.e., words in the free monoid
$X^{*}$ that are minimal among their cyclic rearrangements with respect
to the lexicographical order, and that are not of the form $\lyndonWord^{k}$
for some $\lyndonWord\in X^{*}$ and $k\geq2$. Let $M_{1},M_{2},\ldots,M_{2\nrEdges{\graph}}$
be the matrices obtained from $M=M(\variablesCBC,\variablesL)$ by
multiplying all but one of its rows by~$0$, namely, $[M_{i}]_{jk}=\delta_{ij}[M]_{jk}$,
where $\delta_{ij}$ is the Kronecker delta. For $\lyndonWord=i_{1}i_{2}\ldots i_{\length{\lyndonWord}{}}\in\lyndonWords$,
let $M_{\lyndonWord}=M_{i_{1}}M_{i_{2}}\cdots M_{i_{\length{\lyndonWord}{}}}$.
According to~\cite{FoataZeilberger1999}, Amitsur's identity~\cite{Amitsur1979}
reads 
\[
\det(\identityMatrix{2\nrEdges{\graph}}-M(\variablesCBC,\variablesL))=\det(\identityMatrix{2\nrEdges{\graph}}-(M_{1}+M_{2}+\ldots+M_{2\nrEdges{\graph}}))=\prod_{\lyndonWord\in\lyndonWords}\det(\identityMatrix{2\nrEdges{\graph}}-M_{\lyndonWord}),
\]
and holds as an identity of formal power series in the entries of
$M(\variablesCBC,\variablesL)$. Note that $[M_{\lyndonWord}]_{jk}=0$
for all $j\neq i_{1}$, and that $z_{\lyndonWord}=[M_{\lyndonWord}]_{i_{1}i_{1}}=[M]_{i_{1}i_{2}}[M]_{i_{2}i_{3}}\cdots[M]_{i_{\length{\lyndonWord}{}}i_{1}}$.
In particular, 
\[
\det(\identityMatrix{2\nrEdges{\graph}}-M(\variablesCBC,\variablesL))=\prod_{\lyndonWord\in\lyndonWords}\det(1-z_{\lyndonWord}).
\]
If $\lyndonWord=i_{1}i_{2}\ldots i_{\length{\lyndonWord}{}}\in\lyndonWords$,
then $z_{\lyndonWord}=0$ unless $\path=(i_{1},i_{2},\ldots,i_{\length{\lyndonWord}{}})$
is a circuit in $\graph$, in which case
\[
z_{\lyndonWord}=\weight{\path}\variableCBC{\forward}{\forward}^{\cyclicBumpCount{\forward}{\forward}{\path}}\variableCBC{\backward}{\backward}^{\cyclicBumpCount{\backward}{\backward}{\path}}\variableCBC{\forward}{\backward}^{\cyclicBumpCount{\forward}{\backward}{\path}}\variableCBC{\backward}{\forward}^{\cyclicBumpCount{\backward}{\forward}{\path}}\variableL{\forward}^{\length{\path}{\forward}}\variableL{\backward}^{\length{\path}{\backward}}.
\]
In order to verify~(\ref{eq:Zeta_with_fb_or_bf_bump_count}) and~(\ref{eq:Zeta_with_ff_or_bb_bump_count}),
we consider the $(2\nrEdges{\graph}+\nrVertices{\graph})\times(2\nrEdges{\graph}+\nrVertices{\graph})$
block matrices 
\[
K=\left(\begin{array}{cc}
\identityMatrix{2\nrEdges{\graph}} & G\\
0 & z\identityMatrix{\nrVertices{\graph}}
\end{array}\right),\qquad L=\left(\begin{array}{cc}
E & -G\\
-\tailIncidenceMatrix & \identityMatrix{\nrVertices{\graph}}
\end{array}\right),\qquad K'=\left(\begin{array}{cc}
\identityMatrix{2\nrEdges{\graph}} & FG\\
0 & z\identityMatrix{\nrVertices{\graph}}
\end{array}\right),
\]
where $z\in\mathbb{C}$, $E,F\in\mathbb{\mathbb{C}}^{2\nrEdges{\graph}\times2\nrEdges{\graph}}$,
and $G\in\mathbb{\mathbb{C}}^{2\nrEdges{\graph}\times\nrVertices{\graph}}$.
Since $\det(K)=\det(K')$, we have $\det(KL)=\det(LK')$. The latter
are given by
\[
\det(KL)=\det\left(\begin{array}{cc}
E-G\tailIncidenceMatrix & 0\\
-z\tailIncidenceMatrix & z\identityMatrix{\nrVertices{\graph}}
\end{array}\right)=z^{\nrVertices{\graph}}\det(E-G\tailIncidenceMatrix)
\]
and, provided that $EF=z\identityMatrix{2\nrEdges{\graph}}$,
\[
\det(LK')=\left(\begin{array}{cc}
E & (EF-z\identityMatrix{2\nrEdges{\graph}})G\\
-\tailIncidenceMatrix & z\identityMatrix{\nrVertices{\graph}}-\tailIncidenceMatrix FG
\end{array}\right)=\det(E)\det(z\identityMatrix{\nrVertices{\graph}}-\tailIncidenceMatrix FG).
\]
We choose
\[
G=(\variableL{\forward}\edgeBumpMatrix{\forward}{\backward}+\variableL{\backward}\edgeBumpMatrix{\backward}{\forward})\tailIncidenceMatrix^{T},\qquad\text{so that}\qquad G\tailIncidenceMatrix=\variableL{\forward}\edgeNonBumpMatrix{\forward}+\variableL{\backward}\edgeNonBumpMatrix{\backward}.
\]
In order to show~(\ref{eq:Zeta_with_fb_or_bf_bump_count}), we let
$\variableProduct{\forward}=\variableL{\forward}(\variableCBC{\forward}{\backward}-1)$,
$\variableProduct{\backward}=\variableL{\backward}(\variableCBC{\backward}{\forward}-1)$,
$z=1-\variableProduct{\forward}\variableProduct{\backward}$, 
\[
E=\identityMatrix{2\nrEdges{\graph}}-\variableProduct{\forward}\edgeBumpMatrix{\forward}{\backward}-\variableProduct{\backward}\edgeBumpMatrix{\backward}{\forward},\qquad\text{and}\qquad F=\identityMatrix{2\nrEdges{\graph}}+\variableProduct{\forward}\edgeBumpMatrix{\forward}{\backward}+\variableProduct{\backward}\edgeBumpMatrix{\backward}{\forward}.
\]
 Since $\edgeBumpMatrix{\forward}{\backward}^{2}=\edgeBumpMatrix{\backward}{\forward}^{2}=0$
and $\edgeBumpMatrix{\forward}{\backward}\edgeBumpMatrix{\backward}{\forward}+\edgeBumpMatrix{\backward}{\forward}\edgeBumpMatrix{\forward}{\backward}=\identityMatrix{2\nrEdges{\graph}}$,
we have
\[
EF=\identityMatrix{2\nrEdges{\graph}}-\variableProduct{\forward}\variableProduct{\backward}(\edgeBumpMatrix{\forward}{\backward}\edgeBumpMatrix{\backward}{\forward}+\edgeBumpMatrix{\backward}{\forward}\edgeBumpMatrix{\forward}{\backward})=z\identityMatrix{2\nrEdges{\graph}}.
\]
Moreover, $E-G\tailIncidenceMatrix$ equals $\identityMatrix{2\nrEdges{\graph}}-M(\variablesCBC,\variablesL)$
with $\variableCBC{\forward}{\forward}=\variableCBC{\backward}{\backward}=1$,
and 
\[
\begin{aligned}\tailIncidenceMatrix FG & =\variableL{\forward}\tailIncidenceMatrix\edgeBumpMatrix{\forward}{\backward}\tailIncidenceMatrix^{T}+\variableL{\backward}\tailIncidenceMatrix\edgeBumpMatrix{\backward}{\forward}\tailIncidenceMatrix^{T}+\variableProduct{\forward}\variableL{\backward}\tailIncidenceMatrix\edgeBumpMatrix{\forward}{\backward}\edgeBumpMatrix{\backward}{\forward}\tailIncidenceMatrix^{T}+\variableProduct{\backward}\variableL{\forward}\tailIncidenceMatrix\edgeBumpMatrix{\backward}{\forward}\edgeBumpMatrix{\forward}{\backward}\tailIncidenceMatrix^{T}\\
 & =\variableL{\forward}\weightMatrix{\graph}+\variableL{\backward}\weightMatrix{\graph}^{*}+\variableProduct{\forward}\variableL{\backward}\outdegreeMatrix{\graph}+\variableProduct{\backward}\variableL{\forward}\indegreeMatrix{\graph}.
\end{aligned}
\]
If we number the edges of $\bidirectional{\graph}$ so that each edge
of $\graph$ is followed by its added reverse, then $E$ turns into
a block-diagonal matrix with $\nrEdges{\graph}$ blocks of the form
\begin{equation}
\left(\begin{array}{cc}
1 & -\variableProduct{\forward}\weight{\edge}\\
-\variableProduct{\backward}\weight{\edge}^{-1} & 1
\end{array}\right),\text{ showing that }\det(E)=(1-\variableProduct{\forward}\variableProduct{\backward})^{\nrEdges{\graph}}=z^{\nrEdges{\graph}}.\label{eq:det_E_as_contribution_of_edges}
\end{equation}
In order to show~(\ref{eq:Zeta_with_ff_or_bb_bump_count}), we let
$\variableProduct{\forward}=\variableL{\forward}(\variableCBC{\forward}{\forward}-1)$
and $z=1-\variableProduct{\forward}^{2}$, as well as
\[
E=\identityMatrix{2\nrEdges{\graph}}-\variableProduct{\forward}\edgeBumpMatrix{\forward}{\forward}\qquad\text{and}\qquad F=(1-\variableProduct{\forward}^{2})\identityMatrix{2\nrEdges{\graph}}+\variableProduct{\forward}\edgeBumpMatrix{\forward}{\forward}+\variableProduct{\forward}^{2}\edgeBumpMatrix{\forward}{\forward}^{2}.
\]
Since $\edgeBumpMatrix{\forward}{\forward}^{3}=\edgeBumpMatrix{\forward}{\forward}$,
we have $EF=(1-\variableProduct{\forward}^{2})\identityMatrix{2\nrEdges{\graph}}=z\identityMatrix{2\nrEdges{\graph}}$.
Also, $E-G\tailIncidenceMatrix$ equals $\identityMatrix{2\nrEdges{\graph}}-M(\variablesCBC,\variablesL)$
with $\variableCBC{\backward}{\backward}=\variableCBC{\forward}{\backward}=\variableCBC{\backward}{\forward}=1$,
and
\[
\begin{aligned}\tailIncidenceMatrix FG & =(1-\variableProduct{\forward}^{2})(\variableL{\forward}\tailIncidenceMatrix\edgeBumpMatrix{\forward}{\backward}\tailIncidenceMatrix^{T}+\variableL{\backward}\tailIncidenceMatrix\edgeBumpMatrix{\backward}{\forward}\tailIncidenceMatrix^{T})+\variableProduct{\forward}\variableL{\forward}\tailIncidenceMatrix\edgeBumpMatrix{\forward}{\forward}\edgeBumpMatrix{\forward}{\backward}\tailIncidenceMatrix^{T}+\variableProduct{\forward}^{2}\variableL{\forward}\tailIncidenceMatrix\edgeBumpMatrix{\forward}{\forward}^{2}\edgeBumpMatrix{\forward}{\backward}\tailIncidenceMatrix^{T}\\
 & =(1-\variableProduct{\forward}^{2})(\variableL{\forward}\weightMatrix{\graph}+\variableL{\backward}\weightMatrix{\graph}^{*})+\variableProduct{\forward}\variableL{\forward}\symPart{\degreeMatrix{\graph}}+\variableProduct{\forward}^{2}\variableL{\forward}\symPart{\weightMatrix{\graph}}.
\end{aligned}
\]
Similarly to~(\ref{eq:det_E_as_contribution_of_edges}), each loop
$\edge$ and each pair of mutually reverse edges of $\graph$ contributes
a factor of $1-\variableProduct{\forward}\weight{\edge}$ and $1-\variableProduct{\forward}^{2}$
to $\det(E)$, respectively.
\end{proof}

\section{Invasions\label{sec:Invasions}}

\noindent In this section, we consider unweighted digraphs that may
have parallel edges.
\begin{defn}
An invader is a digraph $\substituent$ with $\nrVertices{\substituent}\geq2$
vertices, the first and last of which are called native. The $\nrVertices{\substituent}-2$
remaining vertices of $\substituent$ are called invasive. If $\graph$
is a digraph, then the invaded digraph $\invadedDiraph{\substituent}{\graph}$
is obtained by replacing each edge $(\vertex,\vertex')$ of $\graph$
with a copy of $\substituent$, where $\vertex$ and $\vertex'$ are
identified with the first and last vertex of $\substituent$, respectively.
\end{defn}
Let $\substituent$ be an invader and let $\substituentTail$ and
$\substituentHead$ denote its first and last vertex, respectively.
Moreover, let $\core$ denote the subdigraph of $\substituent$ induced
by the $\nrVertices{\core}=\nrVertices{\substituent}-2$ invasive
vertices. In particular, $\adjacencyMatrix{\substituent}$ takes the
block form
\[
\adjacencyMatrix{\substituent}=\left(\begin{array}{ccc}
[\adjacencyMatrix{\substituent}]_{\substituentTail\substituentTail} & \rowVector{\substituentTail\rightarrow} & [\adjacencyMatrix{\substituent}]_{\substituentTail\substituentHead}\\
\rowVector{\substituentTail\leftarrow} & \adjacencyMatrix{\core} & \rowVector{\substituentHead\leftarrow}\\
{}[\adjacencyMatrix{\substituent}]_{\substituentHead\substituentTail} & \rowVector{\substituentHead\rightarrow} & [\adjacencyMatrix{\substituent}]_{\substituentHead\substituentHead}
\end{array}\right),
\]
where $\rowVector{\substituentTail\rightarrow},\rowVector{\substituentHead\rightarrow}\in\mathbb{Z}_{\geq0}^{1\times\nrVertices{\core}}$
and $\rowVector{\substituentTail\leftarrow},\rowVector{\substituentHead\leftarrow}\in\mathbb{Z}_{\geq0}^{\nrVertices{\core}\times1}$.
We identify the vertices of $\graph$ with the native ones of $\invadedDiraph{\substituent}{\graph}$.
In particular, we number the vertices of $\invadedDiraph{\substituent}{\graph}$,
starting with vertices of $\graph$ followed by one block of invasive
ones per edge of~$\graph$, so that $\adjacencyMatrix{\invadedDiraph{\substituent}{\graph}}$
takes the block form
\begin{equation}
\adjacencyMatrix{\invadedDiraph{\substituent}{\graph}}=\left(\begin{array}{ccccc}
\adjacencyMatrix{\substituentTail\leftrightarrow\substituentHead} & B_{1} & B_{2} & \cdots & B_{\nrEdges{\graph}}\\
B_{1}^{'} & \adjacencyMatrix{\core} & 0 & \cdots & 0\\
B_{2}^{'} & 0 & \adjacencyMatrix{\core} & \ddots & 0\\
\vdots & \vdots & \ddots & \ddots & \vdots\\
B_{\nrEdges{\graph}}^{'} & 0 & 0 & \cdots & \adjacencyMatrix{\core}
\end{array}\right),\label{eq:Adjacency_matrix_of_invaded_graph}
\end{equation}
where
\begin{equation}
\adjacencyMatrix{\substituentTail\leftrightarrow\substituentHead}=[\adjacencyMatrix{\substituent}]_{\substituentTail t}\outdegreeMatrix{\graph}+[\adjacencyMatrix{\substituent}]_{h\substituentHead}\indegreeMatrix{\graph}+[\adjacencyMatrix{\substituent}]_{\substituentTail\substituentHead}\adjacencyMatrix{\graph}+[\adjacencyMatrix{\substituent}]_{\substituentHead\substituentTail}\adjacencyMatrix{\graph}^{T},\label{eq:First_block}
\end{equation}
and each $B_{i}$ has either one or two possibly nonzero rows, given
by 
\[
\rowVector{\substituentTail\rightarrow}+\rowVector{\substituentHead\rightarrow}\qquad\text{or}\qquad\rowVector{\substituentTail\rightarrow}\text{ and }\rowVector{\substituentHead\rightarrow},
\]
respectively, and similarly for the $B_{i}'$ blocks. Let $\adjointMatrix{x\identityMatrix{\nrVertices{\core}}-\adjacencyMatrix{\core}}$
denote the adjugate matrix of $x\identityMatrix{\nrVertices{\core}}-\adjacencyMatrix{\core}$,
in particular, 
\[
(x\identityMatrix{\nrVertices{\core}}-\adjacencyMatrix{\core})\adjointMatrix{x\identityMatrix{\nrVertices{\core}}-\adjacencyMatrix{\core}}=\det(x\identityMatrix{\nrVertices{\core}}-\adjacencyMatrix{\core})\identityMatrix{\nrVertices{\core}}=\characteristicPolynomial{\core}(x)\identityMatrix{\nrVertices{\core}}.
\]

\begin{prop}
\label{prop:Char_poly_of_invaded_graph}The characteristic polynomial
of $\invadedDiraph{\substituent}{\graph}$ is given by
\begin{eqnarray*}
\characteristicPolynomial{\invadedDiraph{\substituent}{\graph}}(x) & = & \characteristicPolynomial{\core}^{\nrEdges{\graph}-\nrVertices{\graph}}(x)\det\Big(x\characteristicPolynomial{\core}(x)\identityMatrix{\nrVertices{\graph}}\\
 &  & \qquad-\left([\adjacencyMatrix{\substituent}]_{\substituentTail t}\characteristicPolynomial{\core}(x)+p_{\substituentTail\substituentTail}(x)\right)\outdegreeMatrix{\graph}-\left([\adjacencyMatrix{\substituent}]_{h\substituentHead}\characteristicPolynomial{\core}(x)+p_{\substituentHead\substituentHead}(x)\right)\indegreeMatrix{\graph}\\
 &  & \qquad-\left([\adjacencyMatrix{\substituent}]_{\substituentTail\substituentHead}\characteristicPolynomial{\core}(x)+p_{\substituentTail\substituentHead}(x)\right)\adjacencyMatrix{\graph}-\left([\adjacencyMatrix{\substituent}]_{\substituentHead\substituentTail}\characteristicPolynomial{\core}(x)+p_{\substituentHead\substituentTail}(x)\right)\adjacencyMatrix{\graph}^{T}\Big),
\end{eqnarray*}
where
\begin{equation}
p_{\vertex\vertex'}(x)=\sum_{j,k=1}^{\nrVertices{\core}}[\rowVector{\vertex\rightarrow}]_{j}[\adjointMatrix{x\identityMatrix{\nrVertices{\core}}-\adjacencyMatrix{\core}}]_{jk}[\rowVector{\vertex'\leftarrow}]_{k}.\label{eq:Polynomial_factors}
\end{equation}
\end{prop}
\begin{proof}
The polynomial $\characteristicPolynomial{\invadedDiraph{\substituent}{\graph}}(x)\characteristicPolynomial{\core}^{\nrVertices{\graph}}(x)$
equals\setlength{\arraycolsep}{0pt}
\begin{equation}
\begin{array}{l}
\det\left(\begin{array}{cccc}
x\identityMatrix{\nrVertices{\graph}}-\adjacencyMatrix{\substituentTail\leftrightarrow\substituentHead} & -B_{1} & \cdots & -B_{\nrEdges{\graph}}\\
-B_{1}^{'} & x\identityMatrix{\nrVertices{\core}}-\adjacencyMatrix{\core} &  & 0\\
\vdots &  & \ddots & \vdots\\
-B_{\nrEdges{\graph}}^{'} & 0 & \cdots & x\identityMatrix{\nrVertices{\core}}-\adjacencyMatrix{\core}
\end{array}\right)\det\left(\begin{array}{cccc}
\characteristicPolynomial{\core}(x)\identityMatrix{\nrVertices{\graph}} & 0 & \cdots & 0\\
\adjointMatrix{xI-\adjacencyMatrix{\core}}B_{1}' & \identityMatrix{\nrVertices{\core}} &  & 0\\
\vdots &  & \ddots & \vdots\\
\adjointMatrix{xI-\adjacencyMatrix{\core}}B_{\nrEdges{\graph}}' & 0 & \cdots & \identityMatrix{\nrVertices{\core}}
\end{array}\right)\\
=\det\left(\begin{array}{cccc}
M(x) & -B_{1} & \cdots & -B_{\nrEdges{\graph}}\\
0 & x\identityMatrix{\nrVertices{\core}}-\adjacencyMatrix{\core} &  & 0\\
\vdots &  & \ddots & \vdots\\
0 & 0 & \cdots & x\identityMatrix{\nrVertices{\core}}-\adjacencyMatrix{\core}
\end{array}\right)=\characteristicPolynomial{\core}^{\nrEdges{\graph}}(x)\det(M(x)),
\end{array}\label{eq:Char_poly_of_invaded_graph}
\end{equation}
\setlength{\arraycolsep}{\myArraycolsep}where 
\begin{equation}
M(x)=\characteristicPolynomial{\core}(x)(x\identityMatrix{\nrVertices{\graph}}-\adjacencyMatrix{\substituentTail\leftrightarrow\substituentHead})-\sum_{i=1}^{\nrEdges{\graph}}B_{i}\adjointMatrix{x\identityMatrix{\nrVertices{\core}}-\adjacencyMatrix{\core}}B_{i}'.\label{eq:M_char_poly_of_invaded_graph}
\end{equation}
If the blocks $B_{i}$ and $B_{i}'$ correspond to the edge $(\vertex,\vertex')$
of $\graph$, then the only possibly nonzero entries of $B_{i}\,\adjointMatrix{x\identityMatrix{\nrVertices{\core}}-\adjacencyMatrix{\core}}B_{i}'$
have indices in $\{\vertex,\vertex'\}$. If $\vertex=\vertex'$, we
have
\[
[B_{i}\adjointMatrix{x\identityMatrix{\nrVertices{\core}}-\adjacencyMatrix{\core}}B_{i}']_{\vertex\vertex}=\sum_{j,k=1}^{\nrVertices{\core}}[\rowVector{\substituentTail\rightarrow}+\rowVector{\substituentHead\rightarrow}]_{j}[\adjointMatrix{x\identityMatrix{\nrVertices{\core}}-\adjacencyMatrix{\core}}]_{jk}[\rowVector{\substituentTail\leftarrow}+\rowVector{\substituentHead\leftarrow}]_{k},
\]
 whereas if $\vertex\neq\vertex'$ and $w,w'\in\{\vertex,\vertex'\}$,
we have
\[
[B_{i}\adjointMatrix{x\identityMatrix{\nrVertices{\core}}-\adjacencyMatrix{\core}}B_{i}']_{ww'}=\sum_{j,k=1}^{\nrVertices{\core}}[\rowVector{q(w)\rightarrow}]_{j}[\adjointMatrix{x\identityMatrix{\nrVertices{\core}}-\adjacencyMatrix{\core}}]_{jk}[\rowVector{q(w')\leftarrow}]_{k},
\]
where $q(\vertex)=\substituentTail$ and $q(\vertex')=\substituentHead$.
In any case,
\begin{equation}
\sum_{i=1}^{\nrEdges{\graph}}B_{i}\adjointMatrix{x\identityMatrix{\nrVertices{\core}}-\adjacencyMatrix{\core}}B_{i}'=p_{\substituentTail\substituentTail}(x)\outdegreeMatrix{\graph}+p_{\substituentHead\substituentHead}(x)\indegreeMatrix{\graph}+p_{\substituentTail\substituentHead}(x)\adjacencyMatrix{\graph}+p_{\substituentHead\substituentTail}(x)\adjacencyMatrix{\graph}^{T}.\label{eq:Native_invase_interaction}
\end{equation}
The claim now follows from~(\ref{eq:Char_poly_of_invaded_graph}),
(\ref{eq:M_char_poly_of_invaded_graph}), (\ref{eq:First_block}),
and~(\ref{eq:Native_invase_interaction}).\end{proof}
\begin{cor}
If $\substituent$ is a directed path from $\substituentTail$ to
$\substituentHead$ with $\nrVertices{\substituent}$ vertices, then
\[
\characteristicPolynomial{\invadedDiraph{\substituent}{\graph}}(x)=x^{(\nrVertices{\substituent}-2)(\nrEdges{\graph}-\nrVertices{\graph})}\characteristicPolynomial{\graph}(x^{\nrVertices{\substituent}-1})=x^{(\nrVertices{\substituent}-2)(\nrEdges{\graph}-\nrVertices{\graph})}\det(x^{\nrVertices{\substituent}-1}\identityMatrix{\nrVertices{\graph}}-\adjacencyMatrix{\graph}).
\]
\end{cor}
\begin{proof}
We may assume that\def\arraystretch{1} \setlength{\arraycolsep}{2pt}
\[
\adjacencyMatrix{\substituent}=\left(\begin{array}{c|ccc|c}
[\adjacencyMatrix{\substituent}]_{\substituentTail\substituentTail} &  & \rowVector{\substituentTail\rightarrow} &  & [\adjacencyMatrix{\substituent}]_{\substituentTail\substituentHead}\\
\hline  &  &  & \\
\rowVector{\substituentTail\leftarrow} & \quad & \adjacencyMatrix{\core} & \quad & \rowVector{\substituentHead\leftarrow}\\
 &  &  & \\
\hline {}[\adjacencyMatrix{\substituent}]_{\substituentHead\substituentTail} &  & \rowVector{\substituentHead\rightarrow} &  & [\adjacencyMatrix{\substituent}]_{\substituentHead\substituentHead}
\end{array}\right)=\left(\begin{array}{c|ccc|c}
0 & 1 &  &  & 0\\
\hline  & 0 & 1 & \\
 & \hphantom{\ddots} & \ddots & \ddots\\
 &  &  & 0 & 1\\
\hline 0 &  &  &  & 0
\end{array}\right).
\]
Thus, $\characteristicPolynomial{\core}(x)=x^{\nrVertices{\core}}=x^{\nrVertices{\substituent}-2}$,
and the polynomials in~(\ref{eq:Polynomial_factors}) are given by
$p_{\substituentTail\substituentTail}(x)=0$, $p_{\substituentHead\substituentHead}(x)=0$,
$p_{\substituentHead\substituentTail}(x)=0$, and $p_{\substituentTail\substituentHead}(x)=[\adjointMatrix{x\identityMatrix{\nrVertices{\core}}-\adjacencyMatrix{\core}}]_{1\nrVertices{\core}}=1$,
which shows the claimed statement.\setlength{\arraycolsep}{\myArraycolsep}\end{proof}
\begin{cor}
\label{cor:Char_poly_of_path_invasion}If $\substituent$ is an undirected
path from $\substituentTail$ to $\substituentHead$ with $\nrVertices{\substituent}$
vertices, then
\[
\begin{aligned}\characteristicPolynomial{\invadedDiraph{\substituent}{\graph}}(x)= & U_{\nrVertices{\substituent}-2}^{\nrEdges{\graph}-\nrVertices{\graph}}(x/2)\det\big(xU_{\nrVertices{\substituent}-2}(x/2)\identityMatrix{\nrVertices{\graph}}\\
 & \hphantom{U_{\nrVertices{\substituent}-2}^{\nrEdges{\graph}-\nrVertices{\graph}}(x/2)\det\big(}-U_{\nrVertices{\substituent}-3}(x/2)\left(\outdegreeMatrix{\graph}+\indegreeMatrix{\graph}\right)-\left(\adjacencyMatrix{\graph}+\adjacencyMatrix{\graph}^{T}\right)\big),
\end{aligned}
\]
where $U_{-1}=0$ and $U_{n\geq0}$ is the $n$th Chebyshev polynomial
of the second kind.\end{cor}
\begin{proof}
It is well-known that an undirected path $P$ with $\nrVertices{}$
vertices has characteristic polynomial $\characteristicPolynomial P(x)=U_{\nrVertices{}}(x/2)$.
Hence, $\characteristicPolynomial{\core}(x)=U_{\nrVertices{\core}}(x/2)=U_{\nrVertices{\substituent}-2}(x/2)$.
The claimed statement now follows from \setlength{\arraycolsep}{2pt}
\[
\begin{array}{llllcllll}
p_{\substituentTail\substituentTail}(x) & = & [\adjointMatrix{x\identityMatrix{\nrVertices{\core}}-\adjacencyMatrix{\core}}]_{11} & = & U_{\nrVertices{\core}-1}(x/2) & = & [\adjointMatrix{x\identityMatrix{\nrVertices{\core}}-\adjacencyMatrix{\core}}]_{\nrVertices{\core}\nrVertices{\core}} & = & p_{\substituentHead\substituentHead}(x),\\
p_{\substituentTail\substituentHead}(x) & = & [\adjointMatrix{x\identityMatrix{\nrVertices{\core}}-\adjacencyMatrix{\core}}]_{1\nrVertices{\core}} & = & 1 & = & [\adjointMatrix{x\identityMatrix{\nrVertices{\core}}-\adjacencyMatrix{\core}}]_{\nrVertices{\core}1} & = & p_{\substituentHead\substituentTail}(x).
\end{array}
\]
\setlength{\arraycolsep}{\myArraycolsep}\end{proof}
\begin{cor}
\label{cor:Zeta_equivalence_and_invasion}If $\graph$ and $\mate{\graph}$
are digraphs with $\nrVertices{\graph}=\nrVertices{\mate{\graph}}$,
then the following are equivalent:
\begin{enumerate}
\item $\generalizedCharacteristicPolynomial{\graph}=\generalizedCharacteristicPolynomial{\mate{\graph}}$,\label{enu:Equal_gen_char_polys_of_digraphs}
\item $\generalizedCharacteristicPolynomial{\invadedDiraph{\substituent}{\graph}}=\generalizedCharacteristicPolynomial{\invadedDiraph{\substituent}{\mate{\graph}}}$
for every invader $\substituent$,\label{enu:Equal_gen_char_polys_of_invaded_digraphs}
\item $\characteristicPolynomial{\invadedDiraph{\substituent}{\graph}}=\characteristicPolynomial{\invadedDiraph{\substituent}{\mate{\graph}}}$
for every invader $\substituent$,\label{enu:Equal_char_polys_of_invaded_digraphs}
\item $\characteristicPolynomial{\invadedDiraph{\substituent}{\graph}}=\characteristicPolynomial{\invadedDiraph{\substituent}{\mate{\graph}}}$
for every invader $\substituent$ with $\adjacencyMatrix{\substituent}\in\mathbb{Z}_{>0}^{2\times2}$.\label{enu:Equal_char_polys_of_2x2_invasions}
\end{enumerate}
\end{cor}
\begin{proof}
Since $\generalizedCharacteristicPolynomial{\graph}(x,\variableCBC{}{},0,0,0)=\det(x\identityMatrix{\nrVertices{\graph}}+\variableCBC{}{}\outdegreeMatrix{\graph})$
determines $\nrEdges{\graph}$, Proposition~\ref{prop:Char_poly_of_invaded_graph}
gives (\ref{enu:Equal_gen_char_polys_of_digraphs})$\Rightarrow$(\ref{enu:Equal_char_polys_of_invaded_digraphs}).
As (\ref{enu:Equal_gen_char_polys_of_invaded_digraphs})$\Rightarrow$(\ref{enu:Equal_char_polys_of_invaded_digraphs})$\Rightarrow$(\ref{enu:Equal_char_polys_of_2x2_invasions}),
it thus suffices to show that~(\ref{enu:Equal_char_polys_of_2x2_invasions})$\Rightarrow$(\ref{enu:Equal_gen_char_polys_of_digraphs})
and (\ref{enu:Equal_char_polys_of_invaded_digraphs})$\Rightarrow$(\ref{enu:Equal_gen_char_polys_of_invaded_digraphs}).
If $\adjacencyMatrix{\substituent}\in\mathbb{Z}_{>0}^{2\times2}$,
then $\adjacencyMatrix{\invadedDiraph{\substituent}{\graph}}=\adjacencyMatrix{\substituentTail\leftrightarrow\substituentHead}$
in~(\ref{eq:Adjacency_matrix_of_invaded_graph}) so that (\ref{eq:First_block})
gives 
\[
\characteristicPolynomial{\invadedDiraph{\substituent}{\graph}}(x)=\det(x\identityMatrix{\nrVertices{\graph}}-[\adjacencyMatrix{\substituent}]_{\substituentTail t}\outdegreeMatrix{\graph}-[\adjacencyMatrix{\substituent}]_{h\substituentHead}\indegreeMatrix{\graph}-[\adjacencyMatrix{\substituent}]_{\substituentTail\substituentHead}\adjacencyMatrix{\graph}-[\adjacencyMatrix{\substituent}]_{\substituentHead\substituentTail}\adjacencyMatrix{\graph}^{T}).
\]
As the polynomial $\generalizedCharacteristicPolynomial{\graph}$
is determined by its values on $\mathbb{C}\times\mathbb{Z}_{>0}^{4}$,
it follows that~(\ref{enu:Equal_char_polys_of_2x2_invasions})$\Rightarrow$(\ref{enu:Equal_gen_char_polys_of_digraphs}).
We finish by showing that~(\ref{enu:Equal_char_polys_of_invaded_digraphs})$\Rightarrow$(\ref{enu:Equal_gen_char_polys_of_invaded_digraphs}).
If $\substituent$ are $\mate{\substituent}$ are invaders, and if
the first and last vertex of $\invadedDiraph{\mate{\substituent}}{\substituent}$
are chosen to be the first and last one of $\substituent$, then $\invadedDiraph{\mate{\substituent}}{(\invadedDiraph{\substituent}{\graph})}$
and $\invadedDiraph{\invadedDiraph{(\mate{\substituent}}{\substituent)}}{\graph}$
are isomorphic graphs. Thus, if~(\ref{enu:Equal_char_polys_of_invaded_digraphs})
holds, then for any fixed $\substituent$ and all $\mate{\substituent}$
with $\adjacencyMatrix{\mate{\substituent}}\in\mathbb{Z}_{>0}^{2\times2}$,
\[
\characteristicPolynomial{\invadedDiraph{\mate{\substituent}}{(\invadedDiraph{\substituent}{\graph})}}=\characteristicPolynomial{\invadedDiraph{\invadedDiraph{(\mate{\substituent}}{\substituent)}}{\graph}}=\characteristicPolynomial{\invadedDiraph{\invadedDiraph{(\mate{\substituent}}{\substituent)}}{\mate{\graph}}}=\characteristicPolynomial{\invadedDiraph{\mate{\substituent}}{(\invadedDiraph{\substituent}{\mate{\graph}})}},
\]
which, by virtue of the implication (\ref{enu:Equal_char_polys_of_2x2_invasions})$\Rightarrow$(\ref{enu:Equal_gen_char_polys_of_digraphs}),
shows that (\ref{enu:Equal_char_polys_of_invaded_digraphs})$\Rightarrow$(\ref{enu:Equal_gen_char_polys_of_invaded_digraphs}).\end{proof}
\begin{defn}
An invader $\substituent$ is called symmetric if it has an automorphism
that interchanges its native vertices. If $\graph$ is a graph without
loops and $\substituent$ is symmetric, then the symmetrically invaded
digraph $\invadedGraph{\substituent}{\graph}$ is obtained by replacing
each undirected edge $\{(\vertex,\vertex'),(\vertex',\vertex)\}$
of $\graph$ with a copy of $\substituent$, where $\vertex$ and
$\vertex'$ are identified with the first and last vertex of $\substituent$,
respectively.
\end{defn}
Symmetric invasions are well-defined up to isomorphism. If $\substituent$
is a symmetric invader, then $[\adjacencyMatrix{\substituent}]_{\substituentTail t}=[\adjacencyMatrix{\substituent}]_{\substituentHead\substituentHead}$
and $[\adjacencyMatrix{\substituent}]_{\substituentTail\substituentHead}=[\adjacencyMatrix{\substituent}]_{\substituentHead\substituentTail}$
so that $\adjacencyMatrix{\invadedGraph{\substituent}{\graph}}$ takes
block form with $[\adjacencyMatrix{\substituent}]_{\substituentTail t}\degreeMatrix{\graph}+[\adjacencyMatrix{\substituent}]_{\substituentTail\substituentHead}\adjacencyMatrix{\graph}$
as its first block. Proposition~\ref{prop:Char_poly_of_invaded_graph}
and Corollary~\ref{cor:Char_poly_of_path_invasion} simplify as follows
for loop-free graphs $\graph$ with $\nrUndirectedEdges{\graph}=\nrEdges{\graph}/2$
undirected edges.
\begin{prop}
The characteristic polynomial of $\invadedGraph{\substituent}{\graph}$
is given by
\begin{eqnarray*}
\characteristicPolynomial{\invadedGraph{\substituent}{\graph}}(x) & = & \characteristicPolynomial{\core}^{\nrUndirectedEdges{\graph}-\nrVertices{\graph}}(x)\det\Big(x\characteristicPolynomial{\core}(x)\identityMatrix{\nrVertices{\graph}}-\left([\adjacencyMatrix{\substituent}]_{\substituentTail t}\characteristicPolynomial{\core}(x)+p_{\substituentTail\substituentTail}(x)\right)\degreeMatrix{\graph}\\
 &  & \hspace{4.4cm}-\left([\adjacencyMatrix{\substituent}]_{\substituentTail\substituentHead}\characteristicPolynomial{\core}(x)+p_{\substituentTail\substituentHead}(x)\right)\adjacencyMatrix{\graph}\Big).
\end{eqnarray*}
\end{prop}
\begin{cor}
\label{cor:Char_poly_of_symmetric_path_invasion}If $\substituent$
is an undirected path from $\substituentTail$ to $\substituentHead$
with $\nrVertices{\substituent}$ vertices, then
\[
\characteristicPolynomial{\invadedGraph{\substituent}{\graph}}(x)=U_{\nrVertices{\substituent}-2}^{\nrUndirectedEdges{\graph}-\nrVertices{\graph}}(x/2)\det\left(xU_{\nrVertices{\substituent}-2}(x/2)\identityMatrix{\nrVertices{\graph}}-U_{\nrVertices{\substituent}-3}(x/2)\degreeMatrix{\graph}-\adjacencyMatrix{\graph}\right).
\]

\end{cor}
The graph $\invadedGraph{\substituent}{\graph}$ in the previous corollary
arises by introducing $\nrVertices{\core}=\nrVertices{\substituent}-2$
additional vertices on each edge of $\graph$, and is therefore known
as the $\nrVertices{\core}$th subdivision graph of $\graph$. In
particular,~\cite{Mnuhin1980} contains an independent proof of Corollary~\ref{cor:Char_poly_of_symmetric_path_invasion}.
Similarly to Corollary~\ref{cor:Zeta_equivalence_and_invasion},
we obtain that the conditions~(\ref{enu:generalized_characteristic_polynomials_equal})
to~(\ref{enu:characteristic_polynomials_of_2_by_2_invasions_equal})
in Theorem~\ref{thm:Zeta_equivalent_graphs} are equivalent to each
other.

\section{Markov chains\label{sec:Markov_chains}}

\noindent Let $\graph$ be an unweighted graph that may have parallel
edges. Let $\degreeSequence{\graph}$ denote its degree sequence,
i.e., the sequence of its vertex degrees arranged in non-increasing
order. Recall that
\[
\generalizedCharacteristicPolynomialOfGraph{\graph}(x,\variableCBC{}{},\variableL{})=\det(x\identityMatrix{\nrVertices{\graph}}+\variableCBC{}{}\degreeMatrix{\graph}+\variableL{}\adjacencyMatrix{\graph})
\]
and
\[
\markovChainCharacteristicPolynomialOfGraphs{\graph}(x,a,b)=\det(x\identityMatrix{\nrVertices{\graph}}+((a+b)\identityMatrix{\nrVertices{\graph}}+\degreeMatrix{\graph})^{-1}(\adjacencyMatrix{\graph}+a\identityMatrix{\nrVertices{\graph}})).
\]

\begin{prop}
\label{prop:Determination-of-degree-sequence}Each of the polynomials
$\generalizedCharacteristicPolynomialOfGraph{\graph}$ and $\markovChainCharacteristicPolynomialOfGraphs{\graph}$
determines $\degreeSequence{\graph}$. If $\graph$ is simple, then
for any $c\in\mathbb{C}$ the two-variable polynomial $\markovChainCharacteristicPolynomialOfGraphs{\graph|b=c}$
determines $\degreeSequence{\graph}$.\end{prop}
\begin{proof}
The diagonal entries of $\degreeMatrix{\graph}$ are the roots of
$\generalizedCharacteristicPolynomialOfGraph{\graph}(x,-1,0)$. As
for $\markovChainCharacteristicPolynomialOfGraphs{\graph}$, we note
that $a^{-\nrVertices{\graph}}\markovChainCharacteristicPolynomialOfGraphs{\graph}(ax,a,1-a)=\det(x\identityMatrix{\nrVertices{\graph}}+(\identityMatrix{\nrVertices{\graph}}+\degreeMatrix{\graph})^{-1}(a^{-1}\adjacencyMatrix{\graph}+\identityMatrix{\nrVertices{\graph}}))$,
which converges to $\det(x\identityMatrix{\nrVertices{\graph}}+(\identityMatrix{\nrVertices{\graph}}+\degreeMatrix{\graph})^{-1})$
as $a\to\infty$. The latter polynomial determines the eigenvalues
of $(\identityMatrix{\nrVertices{\graph}}+\degreeMatrix{\graph})^{-1}$,
which in turn determine $\degreeSequence{\graph}$. If $\graph$ is
a simple graph and $c\in\mathbb{C}$, then the coefficient of $x^{\nrVertices{\graph}-1}$
in $\markovChainCharacteristicPolynomialOfGraphs{\graph}(x,a,c)$
is given by 
\[
f(a)=\mathrm{tr}\left(((a+c)\identityMatrix{\nrVertices{\graph}}+\degreeMatrix{\graph})^{-1}(\adjacencyMatrix{\graph}+a\identityMatrix{\nrVertices{\graph}})\right)=\sum_{i=1}^{\nrVertices{\graph}}\frac{a}{a+c+[\degreeMatrix{\graph}]_{ii}}.
\]
In particular, $\graph$ has $\nrVertices{}$ vertices of degree $d$
if and only if the singular part of $f$ as $a\to-(c+d)$ equals $na(a+c+d)^{-1}$,
which implies the last statement.\end{proof}
\begin{cor}
If $\graph$ and $\mate{\graph}$ are simple graphs with $\nrVertices{\graph}=\nrVertices{\mate{\graph}}$,
then the following are equivalent:\negthinspace{}\negthinspace{}
\begin{enumerate}
\item $\generalizedCharacteristicPolynomialOfGraph{\graph}=\generalizedCharacteristicPolynomialOfGraph{\mate{\graph}}$,
\item $\markovChainCharacteristicPolynomialOfGraphs{\graph}=\markovChainCharacteristicPolynomialOfGraphs{\mate{\graph}}$,
\item $\markovChainCharacteristicPolynomialOfGraphs{\graph|b=c}=\markovChainCharacteristicPolynomialOfGraphs{\mate{\graph}|b=c}$
for some $c\in\mathbb{C}$.
\end{enumerate}
\end{cor}
\begin{proof}
In any of the three cases, we have $\degreeSequence{\graph}=\degreeSequence{\mate{\graph}}$
by virtue of Proposition~\ref{prop:Determination-of-degree-sequence}.
In particular, $f(a,b)=\det((a+b)\identityMatrix{\nrVertices{\graph}}+\degreeMatrix{\graph})=\det((a+b)\identityMatrix{\nrVertices{\mate{\graph}}}+\degreeMatrix{\mate{\graph}})$.
The claimed equivalences now follow from the homogeneity of $\generalizedCharacteristicPolynomialOfGraph{\graph}$
and $\generalizedCharacteristicPolynomialOfGraph{\mate{\graph}}$
as well as the identity
\[
f(a,b)\markovChainCharacteristicPolynomialOfGraphs{\graph}(x,a,b)=\det((a+xa+xb)\identityMatrix{\nrVertices{\graph}}+x\degreeMatrix{\graph}+\adjacencyMatrix{\graph})=\generalizedCharacteristicPolynomialOfGraph{\graph}(a+xa+xb,x,1).
\]

\end{proof}

\section{Digraph switchings\label{sec:Digraph-switchings}}

\noindent Let $\graph$ be a simple digraph. Its complement $\complementDiGraph{\graph}$
is the digraph given by $\adjacencyMatrix{\complementDiGraph{\graph}}=\allOnesMatrix{\nrVertices{\graph}}-\identityMatrix{\nrVertices{\graph}}-\adjacencyMatrix{\graph}$,
where $\allOnesMatrix{\nrVertices{\graph}}=\{1\}^{\nrVertices{\graph}\times\nrVertices{\graph}}$
denotes the all-ones matrix.
\begin{prop}
\label{prop:J_cospectrality}The completely generalized characteristic
polynomial\textup{
\[
\generalizedCharacteristicPolynomial{\graph}^{\mathrm{c}}(x,y,\variableCBC{\forward}{},\variableCBC{\backward}{},\variableL{\forward},\variableL{\backward})=\det(x\identityMatrix{\nrVertices{\graph}}+y\allOnesMatrix{\nrVertices{\graph}}+\variableCBC{\forward}{}\outdegreeMatrix{\graph}+\variableCBC{\backward}{}\indegreeMatrix{\graph}+\variableL{\forward}\adjacencyMatrix{\graph}+\variableL{\backward}\adjacencyMatrix{\graph}^{T})
\]
}has $y$-degree one. In particular, zeta-equivalent simple digraphs
with zeta-equivalent complements have the same completely generalized
characteristic polynomial.\end{prop}
\begin{proof}
The main idea goes back to~\cite{JohnsonNewman1980}. If $U$ diagonalizes
the rank-one matrix $\allOnesMatrix{\nrVertices{\graph}}$, then the
conjugate $U(x\identityMatrix{\nrVertices{\graph}}+y\allOnesMatrix{\nrVertices{\graph}}+\variableCBC{\forward}{}\outdegreeMatrix{\graph}+\variableCBC{\backward}{}\indegreeMatrix{\graph}+\variableL{\forward}\adjacencyMatrix{\graph}+\variableL{\backward}\adjacencyMatrix{\graph}^{T})U^{-1}$
has exactly one $y$-dependent entry. Moreover, $\generalizedCharacteristicPolynomial{\graph}^{\mathrm{c}}(1,y,0,0,0,0)=1+\nrVertices{\graph}y$,
which shows the first statement. In~particular, if $x,\variableCBC{\forward}{},\variableCBC{\backward}{},\variableL{\forward}$,
and $\variableL{\backward}$ are given, then the affine function $y\mapsto\generalizedCharacteristicPolynomial{\graph}^{\mathrm{c}}(x,y,\variableCBC{\forward}{},\variableCBC{\backward}{},\variableL{\forward},\variableL{\backward})$
is determined by its values at $0$ and $1$, which in turn are determined
by $\generalizedCharacteristicPolynomial{\graph}$ and $\generalizedCharacteristicPolynomial{\complementDiGraph{\graph}}$.
\end{proof}
We present a method for constructing digraphs as in the second part
of Proposition~\ref{prop:J_cospectrality}. Let $\vertices{\graph}$
denote the vertex set of~$\graph$. If $\vertex,\vertex'\in\vertices{\graph}$
and $\vParts{},\vParts{}'\subseteq\vertices{\graph}$, let $\nrEdgeFromXtoY{\vParts{}}{\vParts{}'}$
denote the number of edges from $\vParts{}$ to $\vParts{}'$, and
let $\nrEdgeFromXtoY{\vertex}{\vParts{}'}=\nrEdgeFromXtoY{\{\vertex\}}{\vParts{}'}$,
$\nrEdgeFromXtoY{\vParts{}}{\vertex'}=\nrEdgeFromXtoY{\vParts{}}{\{\vertex'\}}$,
and $\nrEdgeFromXtoY{\vertex}{\vertex'}=\nrEdgeFromXtoY{\{\vertex\}}{\{\vertex'\}}$.
Following~\cite{Schwenk1974}, a partition $\vertices{\graph}=\vParts 1\sqcup\vParts 2\sqcup\ldots\sqcup\vParts{\nrVParts}$
is said to be equitable if for every $i,j\in\{1,\ldots,\nrVParts\}$
and every $\vertex,\vertex'\in\vParts i$, we have $\nrEdgeFromXtoY{\vertex}{\vParts j}=\nrEdgeFromXtoY{\vertex'}{\vParts j}$
as well as $\nrEdgeFromXtoY{\vParts j}{\vertex}=\nrEdgeFromXtoY{\vParts j}{\vertex'}$.
An~equitable partition corresponds to a numbering of $\vertices{\graph}$
with respect to which $\adjacencyMatrix{\graph}$ takes block form
with blocks that have constant row sums and constant column sums.
If $\vParts{},\vParts{}'\subseteq\vertices{\graph}$, we write $\unlinked{\vParts{}}{\vParts{}'}$,
$\halflinked{\vParts{}}{\vParts{}'}$, or $\fullylinked{\vParts{}}{\vParts{}'}$
to indicate that $\vParts{}$ is unlinked, half-linked, or fully linked
to $\vParts{}'$, meaning that for every $\vertex\in\vParts{}$, we
have $\nrEdgeFromXtoY{\vertex}{\vParts{}'}=0$, $\frac{1}{2}|\vParts{}'|$,
or $|\vParts{}'|$, respectively. Whenever $\vParts{}=\{v\}$ or $\vParts{}'=\{\vertex'\}$,
we omit braces as above. Lastly, if $v\in\vertices{\graph}$ is half-linked
to~$\vParts{}$, i.e., if $\halflinked{\vertex}{\vParts{}}$, then
switching of $(\vertex,\vParts{})$ means replacing the $\frac{1}{2}|\vParts{}|$
existing edges from $\vertex$ to $\vParts{}$ by the $\frac{1}{2}|\vParts{}|$
non-existing ones, and similarly for switching of $(\vParts{},\vertex)$.
\begin{thm}
\label{thm:switching}Let $\graph$ be a simple digraph with $\vertices{\graph}=\vParts{}\sqcup\vParts{}'\sqcup\wParts{}\sqcup\xPart$
such that the subdigraphs induced by $\vParts{}$ and $\vParts{}'$
feature an isomorphism $\subgraphIsomorphism\colon\vParts{}\overset{\sim}{\to}\vParts{}'$,
the subdigraphs induced by $\vParts{}$, $\vParts{}'$, and $\wParts{}$
have equitable partitions, namely, $\vParts{}=\bigsqcup_{i=1}^{\nrVParts}\vParts i$,
$\vParts{}'=\bigsqcup_{i=1}^{\nrVParts}\vParts i'$ with $\vParts i'=\subgraphIsomorphism(\vParts i)$,
and $\wParts{}=\bigsqcup_{k=1}^{\nrWParts}\wParts k$, and for every
\textup{$i,j\in\{1,\ldots,\nrVParts\}$,} $k\in\{1,\ldots,\nrWParts\}$,
and $\xertex\in\xPart$,
\begin{enumerate}
\item $\unlinked{\xertex}{\vParts i}$ or $\fullylinked{\xertex}{\vParts i}$,
and $\unlinked{\vParts i}{\xertex}$ or $\fullylinked{\vParts i}{\xertex}$,\label{enu:edges_between_V_and_X}
\item $\unlinked{\xertex}{\vParts i'}$ or $\fullylinked{\xertex}{\vParts i'}$,
and $\unlinked{\vParts i'}{\xertex}$ or $\fullylinked{\vParts i'}{\xertex}$,
\label{enu:edges_between_VPrime_and_X}
\item there exists $\difference{\vParts i}\in\mathbb{Z}$ such that whenever
$v\in\vParts i$ and $\vertex'\in\vParts i'$, then\label{enu:equal_differences}
\[
\difference{\vParts i}=\nrEdgeFromXtoY{\vertex'}{\xPart}-\nrEdgeFromXtoY{\vertex}{\xPart}=\nrEdgeFromXtoY{\xPart}{\vertex'}-\nrEdgeFromXtoY{\xPart}{\vertex},
\]

\item $\unlinked{\xertex}{\wParts k}$, $\halflinked{\xertex}{\wParts k}$,
or $\fullylinked{\xertex}{\wParts k}$, and $\unlinked{\wParts k}{\xertex}$,
$\halflinked{\wParts k}{\xertex}$, or $\fullylinked{\wParts k}{\xertex}$,\label{enu:edges_between_W_and_X}
\item whenever $\wertex,\wertex'\in\wParts k$, then $\nrEdgeFromXtoY{\wertex}{\xPart}=\nrEdgeFromXtoY{\wertex'}{\xPart}$
and $\nrEdgeFromXtoY{\xPart}{\wertex}=\nrEdgeFromXtoY{\xPart}{\wertex'}$,\label{enu:Outdegrees_indegrees_on_W}
\item $\unlinked{\vParts i\sqcup\vParts i'}{\wParts k}$ or $\fullylinked{\vParts i\sqcup\vParts i'}{\wParts k}$,
and $\unlinked{\wParts k}{\vParts i\sqcup\vParts i'}$ or $\fullylinked{\wParts k}{\vParts i\sqcup\vParts i'}$,
\label{enu:edges_between_Vi_ViPrime_and_W}
\item every  $\vertex\in\vParts i$ satisfies $\fullylinked{\fullylinked{\vertex}{\subgraphIsomorphism(v)}}{\vertex}$
and $\unlinked{\unlinked{\vertex}{\vParts i'\backslash\{\subgraphIsomorphism(v)\}}}{\vertex}$,\label{enu:edges_between_Vi_and_ViPrime}
\item if $i\neq j$ and $\difference{\vParts i}\neq\difference{\vParts j}$,
then $\unlinked{\vParts i}{\vParts j'}$ and $\unlinked{\vParts i'}{\vParts j}$,
or $\fullylinked{\vParts i}{\vParts j'}$ and $\fullylinked{\vParts i'}{\vParts j}$,
and $\unlinked{\vParts i}{\vParts j}$ and $\unlinked{\vParts i'}{\vParts j'}$,
or $\fullylinked{\vParts i}{\vParts j}$ and $\fullylinked{\vParts i'}{\vParts j'}$;\\
if $i\neq j$ and $\difference{\vParts i}=\difference{\vParts j}$,
then $\unlinked{\vParts i}{\vParts j'}$ and $\unlinked{\vParts i'}{\vParts j}$,
or $\fullylinked{\vParts i}{\vParts j'}$ and $\fullylinked{\vParts i'}{\vParts j}$,
or $\halflinked{\vParts i}{\vParts j'}$, $\halflinked{\vParts i'}{\vParts j}$,
$\nrEdgeFromXtoY{\vParts i}{\vertex'}=\frac{1}{2}|\vParts i|$ for
every $\vertex'\in\vParts j'$, and there is an edge from $\vertex\in\vParts i$
to $\vertex'\in\vParts j'$ if and only if there is no edge from $\subgraphIsomorphism(\vertex)\in\vParts i'$
to $\subgraphIsomorphism^{-1}(\vertex')\in\vParts j$.\label{enu:edges_between_Vi_Vj_ViPrime_VjPrime}
\end{enumerate}
Let $\mate{\graph}$ be the simple digraph obtained from $\graph$
by performing all possible switchings of the form $(\xertex,\vParts i\sqcup\vParts i')$,
$(\vParts i\sqcup\vParts i',\xertex)$, $(\xertex,\wParts k)$, and
$(\wParts k,\xertex)$, where $x$ ranges over $\xPart$. Then, $\graph$
and $\mate{\graph}$ are zeta-equivalent and have zeta-equivalent
complements.
\end{thm}
If $\graph$ is a graph and $\vParts{}=\vParts{}'=\varnothing$, then
Theorem~\ref{thm:switching} reduces to GM{*}-switching as introduced
in~\cite{HaemersSpence2004}, which is a special case of Godsil-McKay
switching~\cite{GodsilMcKay1982}. The latter method uses that since
$\wParts{}=\bigsqcup_{k=1}^{\nrWParts}\wParts k$ is equitable, there
exists an invertible matrix $Q$, given below, such that the $(\xertex,\wParts k)$
and $(\wParts k,\xertex)$ switchings can be expressed as $\adjacencyMatrix{\mate{\graph}}=Q\adjacencyMatrix{\graph}Q^{-1}$.
The GM{*}-switching method adds the condition~(\ref{enu:Outdegrees_indegrees_on_W})
to obtain $\degreeMatrix{\mate{\graph}}=Q\degreeMatrix{\graph}Q^{-1}$
for the same matrix $Q$, which implies zeta-equivalence. In contrast,
Theorem~\ref{thm:switching} applies to the graphs in Figure~\ref{fig:Zeta_equivalent_complementary_graphs}
with $\wParts{}=\varnothing$ and 
\[
\vParts 1=\{1,2\},\vParts 1'=\{3,4\},\vParts 2=\{5,6\},\vParts 2'=\{7,8\},\xPart=\{9\},\text{ and }\subgraphIsomorphism(\vertex)=\vertex+2,
\]
but there is no invertible matrix $Q$ with $\adjacencyMatrix{\mate{\graph}}=Q\adjacencyMatrix{\graph}Q^{-1}$
and $\degreeMatrix{\mate{\graph}}=Q\degreeMatrix{\graph}Q^{-1}$.
The existence of zeta-equivalent graphs with this property was questioned
in~\cite{WangLiLuXu2011}. Below, we introduce a matrix $R$ such
that for all $\variableCBC{}{}\in\mathbb{R}$ the affine combination
$Q+\variableCBC{}{}R$ is invertible and conjugates $\variableCBC{}{}\degreeMatrix{\graph}+\adjacencyMatrix{\graph}$
into $\variableCBC{}{}\degreeMatrix{\mate{\graph}}+\adjacencyMatrix{\mate{\graph}}$,
which gives zeta-equivalence. Similarly, Durfee and Martin~\cite{DurfeeMartin2014}
used conjugators that are affine in $\variableCBC{}{}$ to obtain
zeta-equivalent graphs, including the ones in Figure~\ref{fig:Zeta_equivalent_complementary_graphs}.
Incidentally, we even found zeta-equivalent graphs such that any conjugator
with polynomial entries in $\variableCBC{}{}$ is at least quadratic
in $\variableCBC{}{}$.
\begin{proof}
[Proof of Theorem \ref{thm:switching}]For $\nrVertices{},\nrVertices{}'\in\mathbb{Z}_{>0}$,
let $\allOnesMatrix{\nrVertices{}\nrVertices{}'}=\{1\}^{\nrVertices{}\times\nrVertices{}'}$,
$\allOnesMatrix{\nrVertices{}}=\allOnesMatrix{\nrVertices{}\nrVertices{}}$,
and 
\[
Q_{\nrVertices{}}=\frac{2}{\nrVertices{}}\allOnesMatrix{\nrVertices{}}-\identityMatrix{\nrVertices{}}.
\]
Following~\cite{GodsilMcKay1982}, we note that $Q_{\nrVertices{}}\allOnesMatrix{\nrVertices{}\nrVertices{}'}=\allOnesMatrix{\nrVertices{}\nrVertices{}'}Q_{\nrVertices{}'}=\allOnesMatrix{\nrVertices{}\nrVertices{}'}$
and $Q_{\nrVertices{}}=Q_{\nrVertices{}}^{T}=Q_{\nrVertices{}}^{-1}$.
Moreover, if $\block{}{}\in\mathbb{C}^{\nrVertices{}\times\nrVertices{}'}$
has constant row sums and constant column sums, then $Q_{\nrVertices{}}\block{}{}=\block{}{}Q_{\nrVertices{}'}$,
and if $\block{}{}\in\{0,1\}^{2\nrVertices{}\times\nrVertices{}'}$
has $\nrVertices{}$ zeros and $\nrVertices{}$ ones in each of its
columns, then $Q_{2\nrVertices{}}\block{}{}=\allOnesMatrix{2\nrVertices{}\times\nrVertices{}'}-\block{}{}$.
Let
\[
R_{\nrVertices{}}^{\mathrm{sym}}=(R_{\nrVertices{}}^{\mathrm{sym}})^{T}=Q_{\nrVertices{}}-\identityMatrix{\nrVertices{}}=\frac{2}{\nrVertices{}}\allOnesMatrix{\nrVertices{}}-2\identityMatrix{\nrVertices{}}
\]
and
\[
R_{2\nrVertices{}}=\left(\begin{array}{cc}
0 & R_{\nrVertices{}}^{\mathrm{sym}}\\
-R_{\nrVertices{}}^{\mathrm{sym}} & 0
\end{array}\right).
\]
In particular, if $\block{}{}\in\mathbb{C}^{\nrVertices{}\times\nrVertices{}'}$
has constant row sums and constant column sums, then $R_{\nrVertices{}}^{\mathrm{sym}}\block{}{}=\block{}{}R_{\nrVertices{}'}^{\mathrm{sym}}$.
Moreover, we have $R_{\nrVertices{}}^{\mathrm{sym}}\allOnesMatrix{\nrVertices{}\nrVertices{}'}=\allOnesMatrix{\nrVertices{}\nrVertices{}'}R_{\nrVertices{}'}^{\mathrm{sym}}=0$,
and therefore $Q_{2\nrVertices{}}R_{2\nrVertices{}}=(\frac{1}{\nrVertices{}}\allOnesMatrix{2\nrVertices{}}-\identityMatrix{2\nrVertices{}})R_{2\nrVertices{}}=-R_{2\nrVertices{}}$.
Since $R_{2\nrVertices{}}$ is skew-symmetric, its eigenvalues are
purely imaginary. For $t\in\mathbb{R}$, it follows that $Q_{2\nrVertices{}}(Q_{2\nrVertices{}}+\variableCBC{}{}R_{2\nrVertices{}})=\identityMatrix{2\nrVertices{}}-\variableCBC{}{}R_{2\nrVertices{}}$
is invertible, which implies the same statement for $Q_{2\nrVertices{}}+\variableCBC{}{}R_{2\nrVertices{}}$.

We number the vertices of $\graph$ according to the partition
\[
\vertices{\graph}=(\vParts 1\sqcup\vParts 1')\sqcup\ldots\sqcup(\vParts{\nrVParts}\sqcup\vParts{\nrVParts}')\sqcup\wParts 1\sqcup\ldots\sqcup\wParts{\nrWParts}\sqcup\xPart
\]
in such a way that $\subgraphIsomorphism\colon\vParts{}\overset{\sim}{\to}\vParts{}'$
is an order isomorphism. Then, $\adjacencyMatrix{\graph}$ takes the
block form
\[
\adjacencyMatrix{\graph}=\left(\begin{array}{lclccll}
\block{\vvPrimeParts 1}{\vvPrimeParts 1} & \cdots & \block{\vvPrimeParts 1}{\vvPrimeParts{\nrVParts}} & \block{\vvPrimeParts 1}{\wParts 1} & \cdots & \block{\vvPrimeParts 1}{\wParts{\nrWParts}} & \block{\vvPrimeParts 1}{\xPart}\\
\quad\vdots &  & \quad\vdots & \quad\vdots &  & \quad\vdots & \quad\vdots\\
\block{\vvPrimeParts{\nrVParts}}{\vvPrimeParts 1} & \cdots & \block{\vvPrimeParts{\nrVParts}}{\vvPrimeParts{\nrVParts}} & \block{\vvPrimeParts{\nrVParts}}{\wParts 1} & \cdots & \block{\vvPrimeParts{\nrVParts}}{\wParts{\nrWParts}} & \block{\vvPrimeParts{\nrVParts}}{\xPart}\\
\block{\wParts 1}{\vvPrimeParts 1} & \cdots & \block{\wParts 1}{\vvPrimeParts{\nrVParts}} & \block{\wParts 1}{\wParts 1} & \cdots & \block{\wParts 1}{\wParts{\nrWParts}} & \block{\wParts 1}{\xPart}\\
\quad\vdots &  & \quad\vdots &  &  & \quad\vdots & \quad\vdots\\
\block{\wParts{\nrWParts}}{\vvPrimeParts 1} & \cdots & \block{\wParts{\nrWParts}}{\vvPrimeParts{\nrVParts}} & \block{\wParts{\nrWParts}}{\wParts 1} & \cdots & \block{\wParts{\nrWParts}}{\wParts{\nrWParts}} & \block{\wParts{\nrWParts}}{\xPart}\\
\block{\xPart}{\vvPrimeParts 1} & \cdots & \block{\xPart}{\vvPrimeParts{\nrVParts}} & \block{\xPart}{\wParts 1} & \cdots & \block{\xPart}{\wParts{\nrWParts}} & \block{\xPart}{\xPart}
\end{array}\right)
\]
with subblocks of the form\setlength{\arraycolsep}{2pt}
\[
\block{\vvPrimeParts i}{\vvPrimeParts j}=\left(\begin{array}{ll}
\block{\vParts i}{\vParts j} & \block{\vParts i}{\vParts j'}\\
\block{\vParts i'}{\vParts j} & \block{\vParts i'}{\vParts j'}
\end{array}\right),\block{\vvPrimeParts i}Y=\left(\begin{array}{l}
\block{\vParts i}Y\\
\block{\vParts i'}Y
\end{array}\right),\text{ and }\block Y{\vvPrimeParts i}=\left(\begin{array}{lc}
\block Y{\vParts i} & \block Y{\vParts i'}\end{array}\right).
\]
\setlength{\arraycolsep}{\myArraycolsep}Blocks involving $\xPart$
are constrained by conditions~(\ref{enu:edges_between_V_and_X})
to~(\ref{enu:Outdegrees_indegrees_on_W}), blocks involving~$\vParts{}$
and $\wParts{}$ are constrained by~(\ref{enu:edges_between_Vi_ViPrime_and_W}),
and blocks involving only $\vParts{}$ are constrained by~(\ref{enu:edges_between_Vi_and_ViPrime})~and~(\ref{enu:edges_between_Vi_Vj_ViPrime_VjPrime}).
More precisely, (\ref{enu:edges_between_V_and_X}) and (\ref{enu:edges_between_VPrime_and_X})
say that any column of $\block{\vParts i}{\xPart}$ or $\block{\vParts i}{'\xPart}$
and any row of $\block{\xPart}{\vParts i}$ or $\block{\xPart}{\vParts i'}$
is either the zero vector or the all-ones vector. Similarly, (\ref{enu:edges_between_W_and_X})
says that any column of $\block{\wParts k}{\xPart}$ and any row of
$\block{\xPart}{\wParts k}$ contains $0$, $\frac{1}{2}|\wParts k|$,
or $|\wParts k|$ ones, with zeros in the other entries. Also, (\ref{enu:edges_between_Vi_ViPrime_and_W})
says that $\block{\vvPrimeParts i}{\wParts k},\block{\wParts k}{\vvPrimeParts i}^{T}\in\{0,\allOnesMatrix{2|\vParts i|\times|\wParts k|}\}$.
With regard to $\block{\vParts i}{\vParts j}$ and $\block{\wParts k}{\wParts l}$,
recall that $\bigsqcup_{i=1}^{\nrVParts}\vParts i$ and $\bigsqcup_{k=1}^{\nrWParts}\wParts k$
are equitable so that these matrices have constant row sums and constant
column sums. By virtue of $\subgraphIsomorphism\colon\vParts{}\overset{\sim}{\to}\vParts{}'$,
we moreover have $\block{\vParts i}{\vParts j}=\block{\vParts i'}{\vParts j'}$.
Condition~(\ref{enu:edges_between_Vi_and_ViPrime}) can be rephrased
as $\block{\vParts i}{\vParts i'}=\block{\vParts i'}{\vParts i}=\identityMatrix{|\vParts i|}$,
and (\ref{enu:edges_between_Vi_Vj_ViPrime_VjPrime}) says that if
$i\neq j$, then $\block{\vParts i}{\vParts j'}$ and $\block{\vParts i'}{\vParts j}$
have the same constant row sums and the same constant column sums,
and $\block{\vParts i}{\vParts j'}+\block{\vParts i'}{\vParts j}$
is a multiple of $\allOnesMatrix{|\vParts i|\times|\vParts j|}$,
regardless of which case applies. Hence, $\block{\vvPrimeParts i}{\vvPrimeParts j}$,
$\block{\vvPrimeParts i}{\wParts k}$, and $\block{\wParts k}{\vvPrimeParts j}$
have constant row sums and constant column sums. Thus, $\outdegreeMatrix{\graph}$
is block-diagonal with $\vParts i\vParts i$- and $\vParts i'\vParts i'$-blocks
of the form $[\outdegreeMatrix{\graph}]_{\vParts i\vParts i}=\outDegree{\vParts i}\identityMatrix{|\vParts i|}$
and $[\outdegreeMatrix{\graph}]_{\vParts i'\vParts i'}=(\outDegree{\vParts i}+\difference{\vParts i})\identityMatrix{|\vParts i|}$,
respectively, and similarly for~$\indegreeMatrix{\graph}$. Equally,
each $[\outdegreeMatrix{\graph}]_{\wParts k\wParts k}$ or $[\indegreeMatrix{\graph}]_{\wParts k\wParts k}$
is a multiple of $\identityMatrix{|\wParts k|}$ by virtue of~(\ref{enu:Outdegrees_indegrees_on_W}).
We define block-diagonal matrices 
\begin{eqnarray*}
Q & = & \diagonalMatrix{Q_{2|\vParts 1|},\ldots,Q_{2|\vParts{\nrVParts}|},Q_{|\wParts 1|},\ldots,Q_{|\wParts{\nrWParts}|},\identityMatrix{|\xPart|}},\\
R & = & \frac{1}{4}\diagonalMatrix{\difference{\vParts 1}R_{2|\vParts 1|},\ldots,\difference{\vParts{\nrVParts}}R_{2|\vParts{\nrVParts}|},0,\ldots,0,0}.
\end{eqnarray*}
For $t\in\mathbb{R}$, the diagonal blocks of $Q+\variableCBC{}{}R$
are invertible, which is why $Q+\variableCBC{}{}R$ is invertible.
Performing all switchings of the form $(\xertex,\vParts i\sqcup\vParts i')$,
where $\xertex$ ranges over $\xPart$, corresponds to replacing $\block{\xPart}{\vvPrimeParts i}$
by $\block{\xPart}{\vvPrimeParts i}Q_{2|\vParts i|}$, and similarly
for $(\vParts i\sqcup\vParts i',\xertex)$, $(\xertex,\wParts k)$,
and $(\wParts k,\xertex)$. Taken all together, we have $Q\adjacencyMatrix{\graph}Q=\adjacencyMatrix{\mate{\graph}}$,
which gives $Q\adjacencyMatrix{\graph}^{T}Q=\adjacencyMatrix{\mate{\graph}}^{T}$
by virtue of $Q=Q^{T}=Q^{-1}$. Moreover, $[\outdegreeMatrix{\graph}]_{\vParts i\vParts i}=[\outdegreeMatrix{\mate{\graph}}]_{\vParts i'\vParts i'}$
and $[\outdegreeMatrix{\graph}]_{\vParts i'\vParts i'}=[\outdegreeMatrix{\mate{\graph}}]_{\vParts i\vParts i}$.
We obtain that $R\outdegreeMatrix{\graph}=\outdegreeMatrix{\mate{\graph}}R$
since their $\vvPrimeParts i\vvPrimeParts i$-blocks are given by\setlength{\arraycolsep}{2pt}
\[
\begin{array}{l}
{\displaystyle \frac{\difference{\vParts i}}{4}}\left(\begin{array}{cc}
0 & R_{|\vParts i|}^{\mathrm{sym}}\\
-R_{|\vParts i|}^{\mathrm{sym}} & 0
\end{array}\right)\left(\begin{array}{cc}
\outDegree{\vParts i}\identityMatrix{|\vParts i|} & 0\\
0 & (\outDegree{\vParts i}+\difference{\vParts i})\identityMatrix{|\vParts i|}
\end{array}\right)\\
\hspace{35mm}={\displaystyle \frac{\difference{\vParts i}}{4}}\left(\begin{array}{cc}
(\outDegree{\vParts i}+\difference{\vParts i})\identityMatrix{|\vParts i|} & 0\\
0 & \outDegree{\vParts i}\identityMatrix{|\vParts i|}
\end{array}\right)\left(\begin{array}{cc}
0 & R_{|\vParts i|}^{\mathrm{sym}}\\
-R_{|\vParts i|}^{\mathrm{sym}} & 0
\end{array}\right).
\end{array}
\]
Similarly, we obtain $R\indegreeMatrix{\graph}=\indegreeMatrix{\mate{\graph}}R$.
Next, we consider the block-diagonal matrix $Q\outdegreeMatrix{\graph}-\outdegreeMatrix{\mate{\graph}}Q$.
We note that $[\outdegreeMatrix{\graph}]_{\wParts k\wParts k}=\outDegree{\wParts k}\identityMatrix{|\wParts k|}=[\outdegreeMatrix{\mate{\graph}}]_{\wParts k\wParts k}$
as well as $[\outdegreeMatrix{\graph}]_{\xPart\xPart}=[\outdegreeMatrix{\mate{\graph}}]_{\xPart\xPart}$.
In particular, the $\wParts k\wParts k$-blocks and the $\xPart\xPart$-block
of $Q\outdegreeMatrix{\graph}-\outdegreeMatrix{\mate{\graph}}Q$ are
zero, whereas its $\vvPrimeParts i\vvPrimeParts i$-blocks are given
by
\[
\begin{array}{l}
{\displaystyle \frac{1}{|\vParts i|}}\left(\allOnesMatrix{2|\vParts i|}[\outdegreeMatrix{\graph}]_{\vvPrimeParts i\vvPrimeParts i}-[\outdegreeMatrix{\mate{\graph}}]_{\vvPrimeParts i\vvPrimeParts i}\allOnesMatrix{2|\vParts i|}\right)-[\outdegreeMatrix{\graph}]_{\vvPrimeParts i\vvPrimeParts i}+[\outdegreeMatrix{\mate{\graph}}]_{\vvPrimeParts i\vvPrimeParts i}\\
={\displaystyle \frac{1}{|\vParts i|}}\left(\begin{array}{cc}
-\difference{\vParts i}\allOnesMatrix{|\vParts i|} & 0\\
0 & \difference{\vParts i}\allOnesMatrix{|\vParts i|}
\end{array}\right)+\left(\begin{array}{cc}
\difference{\vParts i}\identityMatrix{|\vParts i|} & 0\\
0 & -\difference{\vParts i}\identityMatrix{|\vParts i|}
\end{array}\right)={\displaystyle \frac{\difference{\vParts i}}{2}}\left(\begin{array}{cc}
-R_{|\vParts i|}^{\mathrm{sym}} & 0\\
0 & R_{|\vParts i|}^{\mathrm{sym}}
\end{array}\right).
\end{array}
\]
\setlength{\arraycolsep}{\myArraycolsep}Similarly for $Q\indegreeMatrix{\graph}-\indegreeMatrix{\mate{\graph}}Q$.
Finally, we consider $R\adjacencyMatrix{\graph}-\adjacencyMatrix{\mate{\graph}}R$.
Recall that the nonzero subblocks of $R$, namely $R_{|\vParts i|}^{\mathrm{sym}}$,
annihilate all-ones vectors when multiplied from left or right. In
combination with $[R]_{\wParts k\wParts k}=0$ and $[R]_{\xPart\xPart}=0$,
we obtain that any block of $R\adjacencyMatrix{\graph}-\adjacencyMatrix{\mate{\graph}}R$
which involves a $\wParts k$ or $\xPart$ is zero. It therefore suffices
to consider its $\vvPrimeParts i\vvPrimeParts j$-blocks given by
\[
\frac{1}{4}\left(\begin{array}{ll}
R_{|\vParts i|}^{\mathrm{sym}} & 0\\
0 & R_{|\vParts i|}^{\mathrm{sym}}
\end{array}\right)\left(\begin{array}{ll}
\hphantom{-}\difference{\vParts i}\block{\vParts i'}{\vParts j}+\difference{\vParts j}\block{\vParts i}{\vParts j'} & \hphantom{-}\difference{\vParts i}\block{\vParts i'}{\vParts j'}-\difference{\vParts j}\block{\vParts i}{\vParts j}\\
-\difference{\vParts i}\block{\vParts i}{\vParts j}+\difference{\vParts j}\block{\vParts i'}{\vParts j'} & -\difference{\vParts i}\block{\vParts i}{\vParts j'}-\difference{\vParts j}\block{\vParts i'}{\vParts j}
\end{array}\right),
\]
where we used that the blocks of $\block{\vvPrimeParts i}{\vvPrimeParts j}$
have constant row sums and constant column sums. Since $R_{|\vParts i|}^{\mathrm{sym}}$
annihilates multiples of $\allOnesMatrix{|\vParts i|\times|\vParts j|}$,
this matrix is zero whenever $i\neq j$, regardless of which case
in~(\ref{enu:edges_between_Vi_Vj_ViPrime_VjPrime}) applies. For
$i=j$, we use that $\block{\vParts i}{\vParts i}=\block{\vParts i'}{\vParts i'}$
and $\block{\vParts i}{\vParts i'}=\block{\vParts i'}{\vParts i}=\identityMatrix{|\vParts i|}$
to obtain that
\begin{eqnarray*}
[R\adjacencyMatrix{\graph}-\adjacencyMatrix{\mate{\graph}}R]_{\vvPrimeParts i\vvPrimeParts i} & = & {\displaystyle \frac{\difference{\vParts i}}{2}}\left(\begin{array}{cc}
R_{|\vParts i|}^{\mathrm{sym}} & 0\\
0 & -R_{|\vParts i|}^{\mathrm{sym}}
\end{array}\right)\\
 & = & -[Q\outdegreeMatrix{\graph}-\outdegreeMatrix{\mate{\graph}}Q]_{\vvPrimeParts i\vvPrimeParts i}=-[Q\indegreeMatrix{\graph}-\indegreeMatrix{\mate{\graph}}Q]_{\vvPrimeParts i\vvPrimeParts i}.
\end{eqnarray*}
Similarly, $R\adjacencyMatrix{\graph}^{T}-\adjacencyMatrix{\mate{\graph}}^{T}R=R\adjacencyMatrix{\graph}-\adjacencyMatrix{\mate{\graph}}R=-(Q\outdegreeMatrix{\graph}-\outdegreeMatrix{\mate{\graph}}Q)=-(Q\indegreeMatrix{\graph}-\indegreeMatrix{\mate{\graph}}Q)$.

For $\boldsymbol{z}=(x,y,\variableCBC{\forward}{},\variableCBC{\backward}{},\variableL{\forward},\variableL{\backward})\in\mathbb{C}^{6}$,
define 
\[
L_{\graph}(\boldsymbol{z})=x\identityMatrix{\nrVertices{\graph}}+y\allOnesMatrix{\nrVertices{\graph}}+\variableCBC{\forward}{}\outdegreeMatrix{\graph}+\variableCBC{\backward}{}\indegreeMatrix{\graph}+\variableL{\forward}\adjacencyMatrix{\graph}+\variableL{\backward}\adjacencyMatrix{\graph}^{T},
\]
and likewise for $L_{\mate{\graph}}(\boldsymbol{z})$, so that $\generalizedCharacteristicPolynomial{\graph}^{\mathrm{c}}(\boldsymbol{z})=\det(L_{\graph}(\boldsymbol{z}))$
and $\generalizedCharacteristicPolynomial{\mate{\graph}}^{\mathrm{c}}(\boldsymbol{z})=\det(L_{\mate{\graph}}(\boldsymbol{z}))$,
respectively. We note that $Q\allOnesMatrix{\nrVertices{\graph}}=\allOnesMatrix{\nrVertices{\graph}}=\allOnesMatrix{\nrVertices{\graph}}Q$
and $R\allOnesMatrix{\nrVertices{\graph}}=0=\allOnesMatrix{\nrVertices{\graph}}R$
to obtain that for every $\variableCBC{}{}\in\mathbb{C}$
\[
(Q+\variableCBC{}{}R)L_{\graph}(\boldsymbol{z})-L_{\mate{\graph}}(\boldsymbol{z})(Q+\variableCBC{}{}R)=(\variableCBC{}{}(\variableL{\forward}+\variableL{\backward})-(\variableCBC{\forward}{}+\variableCBC{\backward}{}))(R\adjacencyMatrix{\graph}-\adjacencyMatrix{\mate{\graph}}R).
\]
Hence, $\generalizedCharacteristicPolynomial{\graph}^{\mathrm{c}}(\boldsymbol{z})=\generalizedCharacteristicPolynomial{\mate{\graph}}^{\mathrm{c}}(\boldsymbol{z})$
whenever $\variableL{\forward}\neq-\variableL{\backward}$ and $\variableCBC{}{}=(\variableCBC{\forward}{}+\variableCBC{\backward}{})/(\variableL{\forward}+\variableL{\backward})$
is different from the finitely many roots of $\variableCBC{}{}\mapsto\det(Q+\variableCBC{}{}R)$,
which implies $\generalizedCharacteristicPolynomial{\graph}^{\mathrm{c}}(\boldsymbol{z})=\generalizedCharacteristicPolynomial{\mate{\graph}}^{\mathrm{c}}(\boldsymbol{z})$
for all $\boldsymbol{z}\in\mathbb{C}^{6}$.
\end{proof}
\bibliographystyle{amsalpha}
\bibliography{Zeta_equivalent_digraphs}

\end{document}